\documentclass{amsart}

\usepackage[mathscr]{eucal}
\usepackage{amsmath,amssymb}
\usepackage{graphicx}
\usepackage{epstopdf}
\usepackage{multirow}
\usepackage[backref]{hyperref}
\usepackage{hyperref}
\usepackage{enumitem}
\usepackage{tikz}
\usetikzlibrary{matrix}
\usepackage[active]{srcltx}
\usepackage{comment}

\usepackage[utf8]{inputenc}
\usepackage[T1]{fontenc}
\usepackage{lmodern}

\newtheorem{thm}{Theorem}[section]
\newtheorem{THM}{Theorem}
\newtheorem{cor}[thm]{Corollary}
\newtheorem{prop}[thm]{Proposition}
\newtheorem{lemma}[thm]{Lemma}
\theoremstyle{definition}

\newtheorem{remark}[thm]{Remark}

\newcounter{comg}
\definecolor{grass}{rgb}{0.14,0.72,0.2}

\DeclareMathOperator{\Res}{Res}

\DeclareMathOperator{\Sym}{Sym}

\DeclareMathOperator{\tang}{tang}
\DeclareMathOperator{\Aut}{Aut}

\DeclareMathOperator{\sing}{sing}

\DeclareMathOperator{\Aff}{Aff(\mathbb C)}

\DeclareMathOperator{\ord}{ord}

\DeclareMathOperator{\card}{\#}

\DeclareMathOperator{\kod}{kod}
\DeclareMathOperator{\alb}{alb}
\DeclareMathOperator{\Alb}{Alb}
\DeclareMathOperator{\supp}{supp}

\def\Q{\mathbb Q}
\def\R{\mathbb R}
\def\C{\mathbb C}
\def\Corb{C_{\mathrm{orb}}}

\def\calR{\mathcal R}
\def\P{\mathbb P}

\def\h0A{H^0_{\mathrm{inv}}(A,\Omega^1_A)}

\def\ih0A{H^0_{\mathrm{inv}}(A_i,\Omega^1_{A_i})}

\def\Z{\mathbb Z}
\def\F{\mathcal F}

\def\G{\mathcal G}
\def\H{\mathcal H}

\def \Pu{\mathbb P^1}

\def \NG{N{\mathcal G}}
\def \NF{N{\mathcal F}}
\def \KF{K{\mathcal F}}
\def \KR{K{\mathcal R}}
\def \KG{K_{\mathcal G}}
\def \KG{K{\mathcal G}}

\newcounter{bound}

\newenvironment{bound}{\refstepcounter{bound}
\align \tag{\Roman{bound}}}
{\endalign}

\def \cnst{\Upsilon}

\begin{document}

\title
{Toward Effective Liouvillian Integration}
\author[G. Cousin]{Ga\"{e}l Cousin}

\author[A. Lins Neto]{Alcides Lins Neto}

\author[J. V. Pereira]{Jorge Vit\'{o}rio Pereira}
\address{\newline
IMPA, Estrada Dona Castorina, 110, Horto, Rio de Janeiro,
Brasil}\email{gael@impa.br, alcides@impa.br, jvp@impa.br}


\thanks{G. Cousin was supported by
Labex IRMIA, ANR project Iso-Galois and CAPES. A. Lins Neto and J.V. Pereira were partially supported by Cnpq and Faperj.}

\begin{abstract}
We prove that foliations on the projective plane admitting a Liouvillian first integral but not admitting a rational first integral always
have invariant algebraic curves of  degree bounded by a function of the degree of the foliation.
We establish, for the same class of foliations,  the existence
of a bound for the degree of the simplest integrating factor depending only on the degree of the foliation and on the nature of its singularities.
We also prove the existence of invariant algebraic curves of  small degree for foliations with rational first integral and intermediate Kodaira dimension.
\end{abstract}

\maketitle
\setcounter{tocdepth}{1}
\tableofcontents  \sloppy

\section{Introduction}

This paper draws motivation from an ancient question studied by Poincar\'e, Autonne, Painlev\'e and others:  Is it possible to decide if  all the  orbits of  a polynomial vector field on the complex affine plane are algebraic ?

Poincar\'e observes \cite{PoincarePalermoI} that in order to provide a positive answer to the above question it suffices to bound the degree of the general orbit. Even if  Poincar\'e is not explicit on which parameters the bound  should depend on,  in general, examples as simple as  linear vector fields on $\mathbb C^2$  show that such bound must depend  on combinatorial data attached to the singularities of the vector field like their resolution process and  quotients of eigenvalues of the resulting foliation. Soon after the appearance of Poincar\'e's paper, Painlev\'e writes in \cite[pp. 216--217]{zbMATH02673308}  the following paragraph.
 \begin{quotation}
 	J’ajoute qu’on ne peut esp\'erer r\'esoudre d’un coup qui consiste \`a limiter $n$. L'\'enonc\'e vers lequel il faut tendre doit avoir la forme suivante: ``On sait reconna\^itre si l’int\'egrale d’une  \'equation $F(y\prime,y,x)=0$ donn\'ee est alg\'ebrique ou ramener l’\'equation aux quadratures.'' Dans ce dernier cas, la question reviendrait \`a reconna\^itre si une certaine int\'egrale ab\'elienne (de premi\`ere ou de troisi\`eme esp\`ece) n’a que deux ou une p\'eriodes.
 \end{quotation}
 Painlev\'e  suggests that  one should first ask whether or not a
 given polynomial vector field admits a first integral ``expressed through quadratures''; and only then,
 having this special first integral at hand, decide whether or not the leaves are algebraic. To put
 things in perspective   it is useful to notice that the strategy to deal with the analogous problem for linear differential equations with rational coefficients is in accordance with Painlev\'e's suggestion \textit{cf.} \cite{MR527825}.

 The vague terminology  first integral ``expressed through quadratures'' can be formalized in several distinct ways.
 One possible interpretation is  that one should look for
 first integrals belonging to a Liouvillian extension of the differential field
 $(\mathbb C(x,y),\{ \partial_x, \partial_y\})$. For a precise definition and thorough discussion of this concept
 we refer to \cite{MR1062869} and \cite{MR2276503}. An important fact is Singer's theorem  \cite{MR1062869} that asserts that Liouvillian integrable foliations of the plane are \emph{transversely affine} foliations. We recall that a foliation defined by a rational $1$-form $\omega$ is transversely affine if $\omega$ admits an integrating factor, i.e. a closed rational $1$-form $\eta$ such that $d\omega = \omega \wedge \eta$. The class of transversely affine foliations includes the class of virtually transversely Euclidean foliations. These are foliations which, after pull-back by a generically finite rational map, are defined by a closed rational $1$-form.

 More recently, it came to light a family of examples \cite{MR1914932} showing the impossibility of giving bounds for the degree of a general algebraic orbit   depending only on the analytical type of the singularities of the foliation/vector field. They consist on one parameter families of holomorphic foliations
 with fixed analytical type of singularities, such that the general foliation has only finitely many algebraic leaves and, in contrast,  for a dense set of the parameter space the corresponding foliations have algebraic general leaf, of unbounded degree. These families of examples highlight the difficulties of Poincar\'e's original problem, and at the same time provide evidence for the effectiveness of the approach suggested by Painlev\'{e} as, in each family, all the members share a common integrating factor.

 The reader not acquainted to the theory of holomorphic foliations on surfaces is invited to take a look at the first chapters of  the reference textbook
\cite{MR3328860}, see also  \cite{MR2071237}.
\subsection{Invariant algebraic curves of small degree}

 Our main results provide further evidence in favor of Painlev\'e's  strategy.  The first one
 shows that  foliations on the projective plane admitting a Liouvillian first integral, but which do
 not admit a rational first integral, always possess an invariant algebraic  curve of comparatively small degree.

\begin{THM}\label{THM:A}
Let $\mathcal F$ be a foliation of degree $d\ge2$ on the projective plane $\mathbb P^2$.
Assume that $\mathcal F$ admits a Liouvillian first integral but does not admit a rational
first integral. Then $\mathcal F$ admits an algebraic invariant curve of degree at most $12(d-1)$.
Moreover, if $\F$ is not virtually transversely Euclidean then there exists such a curve of degree at most $6(d-1)$.
\end{THM}


The starting point of the proof of Theorem \ref{THM:A}  is Singer's Theorem \cite{MR1062869}. Then, recent results on structure of transversely affine foliations,   \cite{MR3294560} and \cite{LPT},  establish  that a transversely affine foliation on a projective surface is either the pull-back under a rational map of a Riccati foliation; or is birationally equivalent to a  finite quotient of a foliation defined by a closed rational $1$-form. In both cases, our strategy consists in looking for sections of powers of the canonical bundle of the foliation vanishing along invariant algebraic curves. To achieve this in the case of pull-back of  Riccati foliations we explore the description of the positive part of the Zariski decomposition of the canonical bundle  of foliations of Kodaira dimension one due to McQuillan, see \cite{MR2435846} and \cite{MR2071237}, in order to prove  that, for some $k\le 6$, $|\KF^{\otimes k}|$  defines a map to a curve which will contain the sought curves in its fibers. The case of finite quotients of foliations defined by closed rational $1$-forms is trickier and makes use of the nonexistence of rational first integral to guarantee the existence of a non-trivial representation of the fundamental group of the complement of the  polar divisor of the transverse affine structure.
We first deal with foliations  defined by  closed rational $1$-forms. In this case, the non-trivial representation of the fundamental group  allows us to produce logarithmic $1$-forms generically transverse to the foliation and tangent to the zero divisor of the closed rational $1$-form defining it. In the general case, one is asked to understand cyclic quotients of foliations defined by closed rational $1$-forms. The proof goes on by studying the action of the relevant cyclic group on the space of symmetric logarithmic differentials and showing the existence of an invariant  symmetric logarithmic differential of degree $\le 12$ tangent to the sought curves.

Of course, it would be highly desirable to have a similar result for foliations admitting a rational first integral. Unfortunately,
our method to prove Theorem~\ref{THM:A}  exploits extensively the nonexistence of rational first integrals.
Nevertheless, its use can be avoided in the case
of foliations of Kodaira dimension zero or  one. As  a consequence
we obtain a similar result for foliations admitting a rational first integral such that the underlying fibration
is isotrivial with fibers of genus $g\ge 2$; or the underlying fibration has elliptic fibers.

\begin{THM}\label{THM:A2}
Let $\F$ be a foliation on $\mathbb P^2$ admitting a rational first integral and of intermediate Kodaira dimension, i.e.
$\kod(\F) \in \{ 0, 1\}$. Suppose its degree is $d\geq 2$.Then the following assertions hold true.
\begin{enumerate}
\item  If $\F$ is birationally equivalent to an isotrivial fibration of genus $1$ then
  	$\F$ admits an invariant algebraic curve of degree at most $6(d-1)$.
\item  If $\F$ is birationally equivalent to a non-isotrivial fibration of genus $1$ then
  any irreducible algebraic curve invariant by  $\mathcal F$ has degree at most $12(d-1)$.
 \item If $\F$ is birationally equivalent to an isotrivial fibration of genus $g\ge 2$  then	$\F$
admits an invariant algebraic curve of degree at most $42(d-1)$.
\end{enumerate}
\end{THM}

It  seems reasonable to conjecture that  every algebraically integrable foliation of degree $d$ on $\mathbb P^2$ always has an invariant algebraic curve of degree at most
$C_d$ where $C_d$ is a constant depending only on $d$.
Complementary to Theorem~\ref{THM:A2},  \cite{pereira2016effective} provides a bound for the degree of the general leaf of a foliation given by a non-isotrivial fibration of genus bigger than $1$. However, the given bound also depends on the genus of the fiber.

\subsection{Effective bounds for the degree of an integrating factor}
In general, it is not possible to bound the degree of all algebraic curves   invariant by a non-algebraically integrable transversely affine foliation just in function of the degree
of the foliation. Explicit examples, derived from Gauss hypergeometric equation, appeared in \cite{MR1896037}, see also \cite{MR1913040}. The next best thing one might hope for is the existence of a bound for the invariant algebraic curves  depending only on the degree of the foliation and  on the local analytic type of singularities of the foliation. Our  last main result establishes such a bound for a special invariant divisor.

\begin{THM}\label{THM:C}
Let $\F$ be a foliation  on $\mathbb P^2$ with non-algebraic general leaf and  admitting
a Liouvillian first integral. Let $\G$ be a reduction of singularities of $\F$. Then there exists an explicit constant $\cnst$ such that $\F$ admits a transversely affine structure with polar divisor of degree at most~$\cnst$. The constant $\cnst$ is effectively computable and
depends only on the degree of $\F$, $N\G^2$, the number of singularities and the Camacho-Sad indices of $\G$.
\end{THM}

In other words, if a non-algebraically integrable transversely affine foliation $\F$ is defined by a homogeneous $1$-form $\omega$ on $\mathbb \C^3$ then there exists a homogeneous closed rational $1$-form $\eta$ on $\mathbb C^3$ with polar divisor $(\eta)_{\infty}$ having degree bounded by $\cnst$ and satisfying the
equation $d\omega = \omega \wedge \eta$.

In practice, an estimate of $\cnst$ can be explicitly written down if one performs the reduction of singularities of $\F$. The normal bundle and Neron-Severi can be followed step by step, so that $N\G^2$ is easily known. If all the singularities of $\G$ are non-degenerate the Camacho-Sad indices are computed from the first jets of local generating vector fields.  In the presence of a degenerate singularity (saddle-node), the only possible non trivial Camacho-Sad index of an  irreducible analytic separatrix is the so called formal invariant of the singularity. If the singularity has Milnor number $\mu$, this invariant is determined from the $2\mu$-th jet of  a local generating vector field.

Actually, Theorem \ref{THM:C} follows from  an algorithm to decide if a non algebraically integrable foliation admits a transverse affine structure. The algorithm provides \emph{explicitely} an integrating factor/transversely affine structure when one exists.  All this relies on the structure theorem for transversely affine foliations and Theorem~\ref{THM:A}.

\subsection{Acknowledgements} We are grateful to the mathoverflow community in general
and to David Speyer in special for providing valuable and substantial help toward the
proof of Theorem \ref{T:Speyer2}.

\section{Structure of transversely affine foliations}

This section starts by recalling  the basic definitions and properties of transversely affine foliations and then
reviews the structure of this class of foliations following  \cite{MR3294560} and \cite{LPT}.

\subsection{Definition}\label{S:definition}
Let $\F$ be a codimension one holomorphic foliation on a complex manifold $X$ with normal bundle $\NF$, i.e. $\F$ is defined by a holomorphic  section $\omega$ of $\NF\otimes \Omega^1_X$ with zero locus of codimension $\geq 2$ and satisfying $\omega \wedge d \omega =0$. A    singular transverse affine structure for $\mathcal F$ is a meromorphic flat connection
\[
	\nabla :\NF \longrightarrow  \NF \otimes \Omega^1_X(*D), \mbox{ satisfying } \nabla(\omega)=0; \,
\]
where $D$ is a reduced divisor on $X$ and $\Omega^1_X(*D)$ is the sheaf of meromorphic $1$-forms on $X$ with poles (of arbitrary order)
along $D$.
We will always take $D$ minimal, in the sense that the connection form of  $\nabla$ is not holomorphic in any point of $D$. The divisor $D$ is the {singular divisor} of
the transverse affine structure. Taking  multiplicities of the poles as coefficients, we define the polar divisor $(\nabla)_{\infty}$, so that the singular divisor is the support of the polar divisor.

A foliation $\mathcal F$ is a {singular transversely affine foliation} if it admits a singular transverse affine structure. Aiming at simplicity, from now on, when talking about singular transverse  affine structures and singular transversely affine foliations, we will
omit the adjective singular.

When $X$ is an algebraic manifold, the transverse affine structure can be defined by rational $1$-forms.
If $\omega_0$ is a rational $1$-form  defining $\F$  then the existence of a meromorphic flat connection on $\NF$ satisfying $\nabla(\omega)=0$
is equivalent to the existence of a rational $1$-form $\eta_0$ such that
\[
d \omega_0 =  \omega_0 \wedge \eta_0  \quad \text{ and } \quad d \eta_0 = 0.
\]
Indeed, if $U$ is an arbitrary open subset of a complex manifold $X$  where $\NF$ is trivial then a  meromorphic connection on a trivialization of  $\NF$ over $U$
can be expressed as
\[
\nabla_{|U} ( f ) = df +  f \otimes \eta_0 \, ,
\]
where $\eta_0$ is a closed meromorphic $1$-form which belongs to $H^0(U,\Omega^1_X(*D))$. If $\omega_0$ represents $\omega$ in this very same trivialization
then $\nabla_{|U}(\omega_0) = d \omega_0 + \eta_0 \wedge \omega_0$. Therefore  $\nabla(\omega)=0$ is  equivalent to $d \omega_0 =  \omega_0 \wedge \eta_0$.

The equality $d \omega_0 = \omega_0 \wedge \eta_0$ implies that the (multi-valued) $1$-form $ \exp(\int \eta_0) \omega_0$ is closed.
Its primitives are first integrals for the
foliation $\F$. These first integrals belong to a Liouvillian extension of the field of rational functions on $X$, and conversely
the existence of a non-constant Liouvillian first integral for  $\mathcal F$  implies that $\mathcal F$ is transversely affine, see \cite{MR1062869}.

Even if $\omega_0$ and $\eta_0$ may have poles in the complement of $D$, the multi-valued function $\int \exp(\int \eta_0) \omega_0$ coincides with the developing map of $\F_{\vert X-(D\cup \sing\F)}$ and extends holomorphically to the whole of $X-D$. For any given base point $q \in X -D$,  its  monodromy  is an anti-representation $\varrho$ of the fundamental group of the complement of
$D$ in $X$ to the affine group $\Aff = \mathbb C^* \ltimes \mathbb C$. The linear part of $\varrho$ will
be denoted by $\rho$. The abelian representation $\rho$  coincides with the monodromy of the flat meromorphic connection $\nabla$.
\begin{center}
\begin{tikzpicture}
  \matrix (m) [matrix of math nodes,row sep=3em,column sep=4em,minimum width=2em]
  {
    \pi_1(X-D) & \Aff \\
      \, & \mathbb C^* \\};
  \path[-stealth]
    (m-1-1) edge node [above] {$\varrho$} (m-1-2)
            edge node [below] {$\rho$} (m-2-2)
    (m-1-2)  edge node [below] {} (m-2-2)  ;
\end{tikzpicture}
\end{center}
Here and throughout the paper we will deliberately omit the base point of the fundamental groups. Hopefully no confusion will arise.

\subsection{Singular divisor and residues}
Recall from the previous section that the singular divisor of a transverse affine structure $\nabla$ is nothing but
the reduced divisor of poles of $\nabla$. A simple computation shows that the irreducible components of the singular divisor $D$ of a transverse affine structure $\nabla$
for a foliation $\mathcal F$ are invariant  by $\mathcal F$, cf. \cite[Proposition 2.1]{MR3294560}.

Since $\nabla$ is flat we can attach to each irreducible component $C$ of $D$ a complex number $\Res_C(\nabla)$, defined as the residue
of any local meromorphic $1$-form $\eta_0$ defining $\nabla$ at a general point of $C$.  The residues of $\nabla$ determine the Chern
class of $\NF$ as the next proposition shows. For a proof see \cite[Proposition 2.2]{MR3294560}.

\begin{prop}\label{P:residues}
Let $X$ be a projective manifold. If $\nabla$ is any flat meromorphic connection on a line-bundle  $\mathcal L$ then  the class of
$-\sum \Res_{C}(\nabla) [C]$ in $H^2(X,\mathbb C)$,  with  the summation ranging over the irreducible components of the singular divisor $D$, coincides with the Chern
class of $\mathcal L$. Reciprocally, given a $\mathbb C$-divisor $R=\sum \lambda_C C$ with the same class in $H^2(X, \mathbb C)$ as a line bundle $\mathcal L$, there exists a flat meromorphic connection $\nabla_{\mathcal L}$ on
$\mathcal L$ with logarithmic poles and $\Res(\nabla_{\mathcal L})=-R$.
\end{prop}

\subsection{Structure Theorem}
The global structure of transversely affine foliations is described by the next result.

\begin{thm}\label{T:structure}
Let $X$ be a projective manifold and $\mathcal F$ be a singular transversely affine foliation on $X$.   Then at least one of following assertions holds true.
\begin{enumerate}
\item \label{struct1} The foliation is virtually transversely Euclidean, i.e. there exists a generically finite Galois morphism $p:Y\to X$ such that
$p^*\mathcal F$ is defined by a closed rational $1$-form.
\item There exists a transversely affine Ricatti foliation $\mathcal R$ on a surface $S$ and
a rational map $p:X \dashrightarrow S$ such that $p^* \mathcal R = \mathcal F$.
\end{enumerate}
\end{thm}

This result was first established in \cite{MR3294560} under the additional assumption that $H^1(X,\mathbb C)=0$.
The version stated above is proved in \cite{LPT}.

\begin{remark}\label{rem:eucl}
A foliation on an arbitrary projective manifold is virtually transversely Euclidean  if and only if there exists a transversely
affine structure $\nabla$ for $\mathcal F$ which has finite monodromy and at worst logarithmic poles, see \cite[Example 2.10 and proof of Theorem 5.2]{MR3294560}.   If
we restrict to projective manifolds with $H^1(X,\mathbb C)=0$ then the finiteness of the monodromy is equivalent to the rationality of
the residues.
\end{remark}
We also notice the following.
\begin{prop}\label{P:kod0}
Let $\F$ be a reduced  foliation of Kodaira dimension zero on a projective surface $X$.
Then $\F$ is virtually transversely Euclidean.\end{prop}
\begin{proof}
According to \cite[Fact IV.3.3]{MR2435846}, there exists a smooth projective surface $Y$ and a  generically finite rational map $\pi:Y  \dashrightarrow X$ such that
$\G =\pi^* \F$  is defined
by a holomorphic vector field $v$. Consider the one parameter subgroup of $\Aut(X)$ generated by
$v$. Its Zariski closure is an abelian subgroup $H$ of $\Aut(X)$ which preserves $\G$. If every element in the Lie algebra of $H$ is tangent to $\G$ then the leaves of
$\G$ are algebraic and the same
holds true for the leaves of $\F$. Thus we can define $\F$ through a logarithmic $1$-form.  If instead there exists
an element in the Lie algebra
of $H$ generically transverse to $\G$ then $\G$  is also defined
by a closed rational $1$-form, cf. \cite[Corollary 2]{MR1957664}.
In this way, in both cases, $\G$ admits a transversely affine structure given by a logarithmic connection $\nabla_{\G}$ with trivial monodromy.
This transverse structure descends to a transverse structure for $\F$ with
finite monodromy and logarithmic poles, see \cite[Example 2.10]{MR3294560}.
\end{proof}

\subsection{Finite Galois morphism}\label{S:finitegalois}
 As in the sequel we are going to deal exclusively with foliations on projective surfaces,
we will restrict ourselves to the two dimensional case from now on.
Let $\mathcal F$ be a transversely affine foliation on a projective surface $X$ with transversely affine structure given by a logarithmic connection  $\nabla$ with rational residues and finite monodromy. We will give details on the construction the generically finite Galois morphism $p: Y \to X$ appearing in the statement
of the structure Theorem \ref{T:structure}.

Decompose $\Res(\nabla)$ as $\Res_{\mathbb Z} + \Res_{\mathbb Q}$ where $\Res_{\mathbb Z}$ is the round-down of $\Res(\nabla)$ and
$\Res_{\mathbb Q}$ is the fractional part of $\Res(\nabla)$. According to our definition the coefficients of $\Res_{\mathbb Q}$ lie in
$[0,1) \cap \mathbb Q$.

If $\Res_\Q \neq 0$  let $m = m(\nabla)$ be the smallest positive integer such that $m \Res_\Q$ is a divisor with integral coefficients.
Notice that the $m$-th power of the monodromy representation $\rho: \pi_1(U) \to \mathbb C^*$ of $\nabla$, extends to a representation
$\rho^m: \pi_1(X) \to \mathbb C^*$.  If $\rho^m$ is trivial then $m$ is the order of $\rho$, in general  $\ord(\rho) = m \ord(\rho^m)$.
Before proceeding, replace $X$ by the finite \'{e}tale covering determined by $\ker \rho^m$, in order to have the equality $m=\ord(\rho)$.
Once we have done this, we obtain that $\mathcal L= \NF^* \otimes \mathcal O_X( - \Res_\Z)$ is a $m$-th root of $m \Res_\Q$, i.e.
\[
\mathcal L^{\otimes m} = \mathcal O_X(m \Res_\Q) \, .
\]

Let $p : Y \to X$ be the branched covering of degree $m=\ord(\rho)$ determined by $\mathcal L$ branched along the effective divisor $m \Res_\Q$,
see \cite[Chapter I, Section 17]{MR2030225}.  If $\sigma \in H^0(X, \mathcal L^{\otimes m})$ is the section vanishing along $m\Res\Q$ and $L$ (respectively $L_m$) is the total space of $\mathcal L$ (respectively $\mathcal L^{\otimes m}$) then $Y$ is the normalization of the pre-image of the graph of $\sigma$ under the   morphism $\theta_m :L_M \to L$ defined fiberwise by $t\mapsto t^m$. Of course,  $\theta_m$ commutes with the projections to $X$.  Notice  that $p: Y \to X$  is Galois, more precisely there exists $\varphi:  Y \to Y$ an automorphism of order $m$
such that $p \circ \varphi = p$.

The monodromy of this covering coincides with the monodromy of the local system of
flat sections of the connection $\nabla$. An alternative construction of $Y$ is given by taking the \'{e}tale Galois covering  $\tilde U \to U$ determined by
$\ker \rho$  and then compactifying in an equivariant way respecting the compactification $X$ of $U$, see for instance \cite[Theorem 1.3.8]{MR933557}.

Over smooth points of the divisor $\Res_\Q$ the surface $Y$ is smooth, and over normal crossing singularities it has at worst cyclic quotient singularities, see \cite[Chapter III, Theorem 5.2]{MR2030225}.
If we consider the  minimal resolution of singularities then  automorphisms of $Y$ lift to the resolution. Therefore we can replace
$Y$ by its minimal resolution and still denoted by $p: Y \to X$ the projection to $X$ and by $\varphi : Y \to Y$ the automorphism of such projection.

Let $\mathcal G = p^* \mathcal F$. Since  $p$ ramifies only over irreducible components of $\Res_{\Q}$, we have the formula  \cite[Chapter 2, Section 3]{MR3328860}.
\[
\NG = p^* \NF \otimes \mathcal O_Y\left( -E + \sum_{C_0 \subset |\Res_{\Q}|}   (p^*C_0)_{red} - p^* C_0  \right) \, ,
\]
where $E$ is an effective divisor contracted by $p$.
Indeed, if $C$ is an irreducible component of  $p^*\Res_{\Q}$
dominating $C_0 \subset X$ and $\omega \in H^0(X,\Omega^1_X \otimes \NF)$ is a twisted $1$-form defining $\mathcal F$ then
the vanishing order  of $(p^*\omega)$ along $C$ is equal to $q(\nabla,C_0) - 1$  where $q(\nabla,C_0)$ is the order of $\Res_{C_0}(\nabla)$ in the group $\Q/ \Z$.
Therefore the connection $\nabla_{\mathcal G}$ on $\NG$ satisfies
\[
\Res_C(\nabla_{\mathcal G})  = q(\nabla,C_0) \big( \Res_{C_0}(\nabla) + 1 \big) -1 \, .
\]
Moreover
\[
\begin{array}{ccc}
\Res_C(\nabla_{\mathcal G}) < 0 & \text{ if and only if }&\Res_{C_0}(\nabla) \le -1 ;  \quad  \text{ and }\\
\Res_C(\nabla_{\mathcal G}) =-1 & \text{ if and only if }&\Res_{C_0}(\nabla) =-1 ; \quad  \text{ and } \\
\Res_C(\nabla_{\mathcal G}) \ge 0 & \text{ if and only if }&\Res_{C_0 }(\nabla) \ge \displaystyle{\frac{1 - q(\nabla,C_0)   }{q(\nabla,C_0)  }} .\\
\end{array}
\]
By design the covering kills the monodromy of $\nabla$. Consequently  all    residues of $\nabla_{\G}$, including the ones along irreducible components of $E$,  are integers.
Therefore we have a closed rational $1$-form $\omega_{\mathcal G}$ defining $\mathcal G$ with poles over
the irreducible components of $\Res(\nabla)$ with residues smaller than or equal to $-1$; and leaving
invariant the curves over the irreducible components of $\Res(\nabla)$ with residues strictly greater than $-1$.

Notice that the deck transformation $\varphi:Y \to Y$ and the closed rational $1$-form $\omega_{\mathcal G}$ are related through the identity
$\varphi^* \omega_{\mathcal G} = \xi_m \omega_{\mathcal G}$,  for $\xi_m$ an $m$-th root of unity.


We summarize the discussion above in the following proposition.

\begin{prop}\label{P:covering}
Let $\mathcal F$ be a transversely affine foliation on a simply connected projective surface $X$ with logarithmic transversely affine structure $\nabla$
with rational residues. Let $m$ be the  order in $\Q / \Z$ of the group generated by the residues of $\nabla$.
Then there exists a Galois morphism  $ p  : Y \to   X$ with cyclic Galois group from a smooth projective surface $Y$ to $ X$
of degree $m$ such that $p^* \mathcal F$ is defined by a closed rational $1$-form $\alpha$.
Moreover, if $\varphi : Y \to Y$ generates the group of  automorphism of the covering $p$ then  $\varphi^* \alpha = \exp(2 i \pi k /m) \alpha$
for some integer $k$ relatively prime to $m$.
\end{prop}

The restriction to simply-connected surfaces in the result above is necessary only to control
the order $m$ of the monodromy on $\nabla$ in terms of the order in $\mathbb Q/\mathbb Z$ of
the group generated by the residues of $\nabla$.

\begin{lemma}\label{L:fix}
Notation as in Proposition \ref{P:covering}. If we further assume that $\F$ is a foliation with reduced singularities
then  $\KG = p^* \KF$.
\end{lemma}
\begin{proof}
If the Galois morphism $p: Y \to X$ determined by $\nabla$ is a ramified covering, i.e. does not contract curves,
then $\KG = p^* \KF$ according to \cite[Chapter 2, Example 3.4]{MR3328860}.

Let $C_1, \ldots, C_k$ be the irreducible components of the divisor contracted by $p$.
These curves form  a finite disjoint union of Hirzebruch-Jung chains of rational curves which are contracted to singular points
of the support of $\Res_{\Q}$.  Arguing as in  \cite[Chapter 2, Examples 3.1 and  3.4]{MR3328860} we deduce that $\KG =  p^* \KF + E$ for some divisor $E = \sum a_i C_i$ supported on these chains.
Since $\nabla$ is logarithmic, it follows that both  residues of  $\nabla$ along the  separatrices of a saddle node are integers.
Therefore  the singularities of $\F$ at the singularities of the support of $\Res_{\Q}$ are reduced singularities with invertible linear part.
Let $q\in p(C_i)$ be such a singularity. The two local separatrices through $q$  lift to $Y$ as   disjoint separatrices intersecting the Hirzebruch-Jung chains of rational curves on $p^{-1}(q)$
at both of their  extremities. Therefore each of the curves $C_i$  contracted by $p$  contains two singularities of $\G= p^* \F$ with invertible linear part. From this remark and \cite[Chapter $2$, Proposition $2.3$]{MR3328860}, it follows
$E \cdot C = \KG \cdot C = - \chi(C) + Z(\F,C) = 0$, for every  irreducible  curve $C$ contracted by $p$.
We deduce from the negative definiteness of the intersection matrix $(C_i \cdot C_j)$ that  $E=0$ as wanted.
\end{proof}

\section{Virtually transversely Euclidean foliations}\label{S:closed}

This section is devoted to the proof of a particular case of  Theorem \ref{THM:A}. Here
we will investigate transversely affine foliations which do not admit a rational first integral
and are defined by a closed rational $1$-form on a suitable ramified covering
of the initial ambient space. A precise statement is provided in Section  \ref{S:Synthesis} below.

\subsection{Zeros,   poles and invariant algebraic curves of closed rational $1$-forms}
We start things off by observing that the divisor of poles of  closed rational \mbox{$1$-forms}
defining reduced foliations do not intersect the invariant algebraic curves not contained in it.

\begin{lemma}\label{L:zerosandpoles}
Let $\omega$ be a closed rational $1$-form on a projective surface $S$
and consider the foliation $\F$ defined by it.
If $\F$ is reduced in Seidenberg's sense
then any germ of irreducible  $\F$-invariant curve  is either contained in $(\omega)_{\infty}$ or does not intersect it.
In particular,  the zero divisor  $(\omega)_0$ does not intersect the polar divisor $(\omega)_{\infty}$.
\end{lemma}
\begin{proof}	
	Since $\omega$ is closed it is clear that any irreducible component of $(\omega)_0$ or of $(\omega)_{\infty}$ is $\F$-invariant. Thus the last conclusion of the statement follows immediately from the first.

	Let $C$ be a germ of irreducible $\F$-invariant curve not contained in $(\omega)_{\infty}$.
Aiming at a contradiction assume $C$ is a germ centered at  a point $p \in |(\omega)_{\infty}|$.  Notice that $p \in \sing(\F)$.

Since there are at most two germs of $\F$-invariant curves centered at reduced singularity, it turns out that the polar locus is smooth at
$p$. Thus, at a neighborhood of  $p$, we can write $\omega$  as
\[
\lambda \frac{df}{f} + d\left( \frac{g}{f^k} \right)
\]
where $\lambda \in \mathbb C$, $k\in \mathbb N$,   $f$ is an irreducible germ of  holomorphic function cutting out  $(\omega)_{\infty}$, and $g$ is a holomorphic germ
relatively prime to $f$.

A simple computation, using that $p \in \sing(\F)$, shows that $g(p)=0$.
If the linear parts of  $f$ and $g$ at $p$ are proportional (in particular if $g$ has trivial linear part) then the singularity is nilpotent. If instead  the linear parts of $f$ and $g$ at $p$
are linearly independent then we obtain a non-reduced singularity with quotient of eigenvalues equal to $k^{\pm 1}$. All possibilities  contradict our assumptions.
\end{proof}

\subsection{Quasi-abelian varieties and Iitaka-Albanese morphism}
For further use we briefly review  Iitaka’s theory of quasi-Albanese maps \cite{MR0429884}.

Let $X$ be a projective manifold and $D$ a simple normal crossing divisor on $X$. The Iitaka-Albanese variety
of the pair $(X,D)$ is by definition the quotient of $H^0(X, \Omega^1_X(\log D))^*$ by the image of the map
\begin{align*}
H_1(X-D,\mathbb Z) & \longrightarrow H^0(X, \Omega^1_X(\log D))^* \\
\gamma & \mapsto \left( \alpha \mapsto \int_{\gamma} \alpha \right) \, .
\end{align*}
We will denote the Iitaka-Albanese variety by $\Alb(X,D)$. It is a connected abelian algebraic group
which comes with a surjective morphism  to the Albanese variety of $X$ with fibers isomorphic to
algebraic tori $(\mathbb G_m)^r \simeq (\mathbb C^*)^r$ of dimension $r = h^1(X-D, \mathbb C) - h^1(X,\mathbb C)$.

The Iitaka-Albanese variety is an example of a quasi-abelian variety (in the sense of Iitaka). By definition a quasi-abelian variety is an  abelian connected
 algebraic group  which fibers over an abelian variety with algebraic tori as fibers.

The holomorphic  map
\begin{align*}
 X - D & \longrightarrow \Alb(X,D) \\
x &\mapsto \left( \alpha  \mapsto \int_{x_0}^x \alpha \right)
\end{align*}
extends to a rational map $\alb(X,D) : X \dashrightarrow \Alb(X,D)$. This is the
Iitaka-Albanese map of the pair $(X,D)$. It depends on the choice of a base point $x_0 \in X-D$ but
any two choices lead to maps that differ by translations. The image of $\alb(X,D)$ is non-degenerate, i.e. is not contained in any proper subgroup of $\Alb(X,D)$.

One of the key properties of $\alb(X,D)$ is that the pull-back through it of the translation invariant forms on $\Alb(X,D)$
coincides with $H^0(X, \Omega^1_X(\log D))$. It will be convenient to denote the translation invariant holomorphic $1$-forms
on a quasi-abelian variety $A$ by $\h0A$.

Let $A$ be a quasi-abelian variety. By an automorphism of a quasi-abelian variety
we mean a biholomorphic map which respects the group law. We will denote the group of automorphisms of $A$  by $\Aut(A)$ .
 If $G \subset \Aut(A)$ is a finite group we will denote by $A_G$ the set of points in $A$ with non-trivial $G$-isotropy group.
In other words $x \in A$ belongs to $A_G$ if and only if there exists an element $g \in G$ different from the identity such that $g(x)=x$.

The next two results are stated  in \cite{MR1703549} for Abelian varieties.
The proofs easily adapt to the more general case of quasi-abelian varieties.
We start with \cite[Theorem 2.1]{MR1703549}.

\begin{thm}\label{T:autA}
	Let $A$ be a quasi-abelian variety and $G=<\alpha>$ a cyclic group of automorphisms.
	Suppose $1 \le d_1 < d_2 < \ldots < d_r$ are the orders of eigenvalues of $\alpha^* : \h0A \to \h0A$.
	Then there are $G$-stable quasi-abelian subvarieties $A_1, \ldots, A_r$ of $A$ such that
	\begin{enumerate}
		  \item $\alpha_i  =\alpha_{|A_i}$ is of order $d_i$;
		  \item $(A_1)_{<\alpha_1>}$ is $A_1$ if $d_1=1$, otherwise $(A_1)_{<\alpha_1>}$ is finite;
		  \item for $i>1$, $(A_i)_{<\alpha_i>}$ is always finite;
		  \item the addition map $A_1 \times \cdots \times A_r \to A$ is an isogeny.
	\end{enumerate}
\end{thm}

\begin{remark}
The proof of the Theorem \ref{T:autA}  does not rely on Poincar\'{e}'s complete reducibility Theorem.
It constructs the abelian/quasi-abelian subvarieties $A_i$ concretely by looking at fixed points
of iterates of $\alpha$, see \cite{MR1703549} for details.
Beware that the  obvious analogue of Poincar\'{e}'s complete reducibility theorem does not hold for
quasi-abelian varieties in general. There exists a quasi-abelian  variety $A$ containing a quasi-abelian
subvariety $B$ for  which there is no quasi-abelian subvariety $C \subset A$ such that $A$ is isogeneous to
$B \times C$, cf. \cite[Example 2.4]{MR1312575}. Nevertheless if $B \subset A$ is an abelian (not just quasi-abelian)
subvariety of a quasi-abelian variety $A$ then there exists a quasi-abelian subvariety $C$ such that $A$ is isogeneous
to $B \times C$, see \cite[Proposition 2.3]{MR1312575}.
\end{remark}

Denote by $\varphi$ the Euler totient function, i.e. $\varphi(d)$ is the number of primitive roots
of unity of order $d$. The next result corresponds to \cite[Proposition 1.8]{MR1703549}.

\begin{prop}\label{P:autA}
	Suppose $\alpha$ is an automorphism of order $d$ of a quasi-Abelian variety $A$ of dimension $g$.
	If $A_{<\alpha>}$ is finite then the set $\Phi_{\alpha} = \{ \text{eigenvalues of } \alpha^* : \h0A \to \h0A \}$
    contains $\varphi(d)/2$  distinct and  pairwise  non complex conjugate primitive $d$-th roots of unity.
\end{prop}

\subsection{Foliations defined by closed rational $1$-forms} We study first the case of foliations defined by closed rational $1$-forms.

\begin{lemma}\label{L:ntkernel}
Let $X$ be a projective manifold and $D$ be a simple normal crossing divisor on $X$.
Suppose  $C \subset X-D$ is a compact curve  such that the restriction morphism $H^1(X -D, \mathbb C) \to H^1(C, \mathbb C)$
has non-trivial kernel. Then the restriction morphism $H^0(X,\Omega^1_X(\log D)) \to H^0(C,\Omega^1_C)$
also has non-trivial kernel.
\end{lemma}
\begin{proof}
If the kernel of  $H^0(X, \Omega^1_X) \to H^0(C, \Omega^1_C)$  is non-trivial   then there is nothing to prove.
Assume the contrary and notice that the kernel of
$H^1(X,\mathcal O_X) \to H^1(C, \mathcal O_C)$ is also trivial thanks to
the functoriality of  Hodge decomposition $H^1(X,\mathbb C) \simeq H^0(X,\Omega^1_X) \oplus H^1(X, \mathcal O_X)$.

Hodge Theory gives the existence of a  decomposition, as real vector spaces, $H^1(X-D,\mathbb C) = H^0(X, \Omega^1_X(\log D)) \oplus  \sqrt{-1} H^1(X,\mathbb R)$, see for instance \cite[Proposition 3.6]{MR3035328}
\cite[proof of Proposition V.1.4]{MR1487227}.
Therefore the non-injectivity of $H^1(X -D, \mathbb C) \to H^1(C, \mathbb C)$ implies the existence of a logarithmic differential with poles on $D$
and non-zero residues. From the long exact sequence in cohomology deduced from
\[
 0 \to \Omega^1_X \to \Omega^1_X(\log D) \to \oplus \mathcal O_{D_i} \to 0
\]
we infer that the first boundary map $\oplus H^0(D_i, \mathcal O_{D_i}) \to H^1(X,\Omega^1_X)$  also has non-trivial kernel.
Since this boundary map is nothing but the Chern class of the corresponding divisor, we obtain a divisor $E \neq 0$ with vanishing Chern class and
support contained in the support of $D$. If we consider the unitary flat connection
on $\mathcal O_X(E)$ and pull it back to $\mathcal O_X$ through a section $s$ of $\mathcal O_X(E)$ such that $E = (s)_0 - (s)_{\infty}$, we obtain a
logarithmic differential on $X$ with poles on $E$ such that all its periods (including the ones around irreducible components of $E$) are
purely imaginary.  In the notation of \cite[Section 3]{MR2177196}, this is the unique logarithmic  $1$-form $\omega_E$ with $\Res(\omega_E) = E$
and purely imaginary periods. We can apply \cite[Proposition 3.3]{MR2177196} to conclude that $\omega_E$ vanishes identically when pulled-back to
$C$. This implies the non-injectivity of $H^0(X,\Omega^1_X(\log D)) \to H^0(C,\Omega^1_C)$.
\end{proof}

\begin{prop}\label{P:chave}
Let $\mathcal F$ be a reduced foliation on a projective surface $X$ defined by  a closed rational $1$-form $\omega$ with polar divisor $D$.
Assume $\mathcal F$ does not admit a rational first integral and let $C$ be the maximal compact $\mathcal F$-invariant
algebraic curve contained in $X-D$. Then there exists a quasi-abelian variety $A$ and a non-constant morphism from $r: X-D \to A$
with image not contained in the translate of any proper quasi-abelian subvariety
which  contracts all the irreducible components of $C$ to points.
\end{prop}
\begin{proof}
We will construct $r: X - D \to A$ as a quotient of the Iitaka-Albanese morphism of $X- D$.
Consider the representation  of $\pi_1(X-D)$ in $\mathbb C$ defined by integration of $\omega$. Since $\mathcal F$
does not admit a rational first integral this representation must be infinite. Therefore $H^1(X-D,\mathbb C)$ is
infinite as well. Consider the restriction morphism $H^1(X -D, \mathbb C) \to H^1(C, \mathbb C)$. Since $C$ is $\mathcal F$
invariant the representation defined by $\omega$ is in the kernel of this morphism.

Lemma \ref{L:ntkernel} implies that the restriction morphism $H^0(X,\Omega^1_X(\log D)) \to H^0(C,\Omega^1_C)$ has non-trivial kernel.
Dualizing it, we deduce that
\[
H^0(C,\Omega^1_C)^* \to H^0(X,\Omega^1_X(\log D))^*
\]
is not surjective. Therefore the induced morphism
\[
\Alb(C) = \oplus \Alb(C_i) \to \Alb(X,D)
\]
is not surjective.

We define the quasi-abelian variety $A$ as the cokernel of this morphism.
The Iitaka-Albanese morphism $X- D \to \Alb(X,D)$ composed with the natural projection
$\Alb(X,D) \to A$  gives rise to the sought morphism.
\end{proof}

\begin{cor}\label{C:chave}
Notation and assumptions as in Proposition \ref{P:chave}.   If
$V_C \subset H^0(X, \Omega^1_X(\log D))$ is the vector subspace consisting of logarithmic $1$-forms which are
identically zero  when pulled-back to (all the irreducible components of) $C$,
then the  dimension of $V_C$ is at least one. Moreover, if $\omega \in V_C$  then  $\dim V_C \ge 2$.
\end{cor}
\begin{proof}
Let $r : X-D \to A$  be the morphism produced by Proposition \ref{P:chave}.
Let $\h0A$ be the vector space of
translation invariant $1$-forms on $A$.
Since the image of $r$ is not contained
in any proper quasi-abelian subvariety it follows that the pull-back map
$ r^* : \h0A \to H^0(X, \Omega^1_X(\log D)) $
is injective. Moreover, by construction, its image coincides exactly with $V_C$. This is sufficient to prove that $\dim V_C \ge 1$.

Suppose now that $\omega \in V_C$. Observe that  $r$ extends to a rational
map $\overline r: X \dashrightarrow A$.  If the dimension of $A$ is one then
the composition of $\overline r$ with any non-constant rational function $f \in \mathbb C( A)$
would be a rational first integral for $\F$. Since we are assuming that such a first integral does not exist,
we conclude that $\dim V_C\ge 2$ when $\omega \in V_C$.
\end{proof}

\begin{prop}\label{P:cotasimples}
Let $\mathcal F$ be a  reduced foliation on a projective surface $X$
defined by a closed rational $1$-form $\omega$ with polar divisor $D$.
If $\mathcal F$ does not admit a rational first integral then there exists a section
of $\KF$ vanishing along the compact curves contained in $X-D$ and invariant by $\mathcal F$.
\end{prop}
\begin{proof}
According to Corollary \ref{C:chave} there exists a logarithmic $1$-form $\alpha \in H^0(X, \Omega^1_{X}(\log D))$
such that the foliation defined by $\alpha$ is distinct from the foliation defined by $\omega$ and which vanishes when pulled-back to
every compact curve in $X-D$  invariant by $\mathcal F$. If $v \in H^0(X, TX \otimes  \KF)$
is a twisted vector field defining $\mathcal F$ then the contraction of $\alpha$ and $v$ gives a section $\sigma$ of $\KF$
vanishing along the compact curves contained in $X-D$.
\end{proof}

\begin{cor}
Let $\mathcal F$ be a  degree $d$ foliation on $\mathbb P^2$ given by a closed rational $1$-form $\omega$.
If $\mathcal F$ does not admit a rational first integral then the (reduced) support of the  zero divisor of  $\omega$
has degree at most $d-1$.
\end{cor}
\begin{proof}
Let $\pi: (Y, \mathcal G) \to (\mathbb P^2, \mathcal F)$ be a reduction of singularities of $\mathcal F$.
The $1$-form $\pi^* \omega$ defines $\mathcal G$. According to Lemma \ref{L:zerosandpoles} the zero divisor of $\pi^*\omega$ is disjoint from the polar locus of $\pi^* \omega$. Proposition \ref{P:cotasimples} gives a section $\sigma$ of $\KG$ vanishing along the zero divisor of $\pi^* \omega$. This section descends to a section
of $\KF$ vanishing along the zero divisor of $\omega$. To conclude
it suffices to observe that $\KF = \mathcal O_{\mathbb P^2}(d-1)$.
\end{proof}

\subsection{Finding invariant symmetric differentials} To reduce the case of foliations defined by closed rational $1$-forms after a ramified covering to the case just studied, we will now show how to produce logarithmic symmetric differentials vanishing along the compact curves. In the case $\omega \notin V_C$ (notation as in the previous subsection) this is achieved through a direct application of the result below which is due to  D. Speyer, see  \cite{136310}.

\begin{thm}\label{T:Speyer}
Let $2\le m \in \mathbb N$ be a natural number and   $\Phi$ be a subset of the set $P(m)\subset \mathbb C^*$ of primitive $m$-th roots of the unity. If
 $P(m)$ is the disjoint  union of $\Phi$ and  $\Phi^{-1}$ then
\[
 1 \in  \Phi^N = \underbrace{\Phi \cdot   \cdots  \cdot \Phi}_{N \, \mathrm{ times }} \,
\]
for some $N\le 6$.
\end{thm}
\begin{proof}
The complete proof can be found in \cite{136310}. We will not reproduce it
because the proof of Theorem \ref{T:Speyer2} below follows very closely Speyer's arguments.
Here we will only deal with the case $m = 2^k$ for $k\ge 2$ since
it will be used in what follows.

For $m=2^2$ the result is clear. It suffices to take any element of $\Phi$ and raise it
to the $4$-th power.

Suppose now that $m=2^k$ with $k \ge 3$. Notice that the set $\Phi^3 = \Phi \cdot \Phi \cdot \Phi$ is contained in $P(m)$.
Aiming at a contradiction assume that  $\Phi^3 \subset \Phi$. If $a,b  \in \Phi$  then
 $a^2 b \in \Phi$. Hence for any $a \in \Phi$ we deduced that $a^{2l + 1}$ also belongs to $\Phi$. This
shows that $\Phi = P(m)$. Contradiction.

Since $\Phi^3$ is not contained in $\Phi$ we can find three elements in $\Phi$, say $a,b,c$ and
one element $d$ in $\Phi^{-1}$ such that $abc = d$. Since $d^{-1}\in \Phi$ we obtain that $abcd^{-1} =1 $
as wanted.
\end{proof}

\begin{cor}
Let $A$ be a quasi-abelian variety and $\alpha$ be a finite automorphism of $A$. Then
there exists a non-zero holomorphic  section of $\Sym^k \Omega^1_A$ invariant by $\alpha$ for some
$k \le 6$.
\end{cor}
\begin{proof}
If the only eigenvalue of $\alpha^* : \h0A \rightarrow \h0A$ is $d_1=1$ then, by  Theorem \ref{T:autA}, $A_1=A$, $\alpha$ is the identity, and any $1$-form is invariant.
Otherwise, some eigenvalue $d_i$ of $\alpha^*$ is not $1$, Let $A_i \subset A$ be the quasi-abelian subvariety  associated to $d_i$ in Theorem \ref{T:autA}.
By Theorem \ref{T:autA} (3), ${(A_i)}_{<\alpha_i>}$ is finite and, by Proposition \ref{P:autA},  $\alpha_i^*$ has $\varphi(n)/2$ distinct eigenvalues.

Let $\lambda$ be an eigenvalue associated to the eigenvector $\gamma$ of $\alpha_i^* : \ih0A \rightarrow \ih0A$.
Then $\gamma$ extends to an element  $\hat{\gamma} \in \h0A$ such that $\alpha^*\hat{\gamma}=\lambda \hat{\gamma}$. In particular, $\alpha^*$ also possesses $\varphi(n)/2$ distinct eigenvalues and application of Theorem \ref{T:Speyer} produces an $\alpha^*$-invariant symmetric differential of the required degree.
\end{proof}

\begin{thm}\label{T:Speyer2}
Let $m \in \mathbb N$ be a natural number satisfying $\varphi(m) \ge 4$,   $\Phi$ be a subset  of the set $P(m)\subset \mathbb C^*$ of primitive $m$-th roots of the unity,
and $\lambda$ be an element of $P(m)$.
If $P(m) = \Phi \cup \Phi^{-1} \cup  \{ \lambda, \lambda^{-1} \}$ then
\[
 1 \in \Phi^N = \underbrace{\Phi \cdot   \cdots \cdot \Phi}_{N \, \mathrm{ times }} \,
\]
for some $N\le 12$.
\end{thm}
\begin{proof}
If $\varphi(m)=4$ then $m \in \{ 5,8,10,12\}$ and the result is a trivial consequence of  $\xi^m=1$.

It will be convenient to adopt the additive notation.
Let $C(m)=\mathbb Z/m\mathbb Z$ denote the cyclic group with $m$ elements, and $U(m)$ be the set of its generators, i.e. $U(m) = (\mathbb Z/m\mathbb Z)^*$.

For $p$ a prime number we have the exact sequence
\[
0 \to C(p) \to C(mp) \xrightarrow{\pi} C(m) \to 0 \, .
\]
If $y \in U(m)$ then the cardinality of  $\pi^{-1}(y)\cap U(mp)$ is $p$ if $p$ divides $m$, and $p-1$ otherwise.

Let $A \subset U(mp)$ be a set and $\lambda \in U(mp)$ an element such that
\[
U(mp) = A \coprod - A \coprod \{ \lambda \} \coprod \{ - \lambda \} \, .
\]
Notice that, for every $y \in U(m) - \{ \pi(\lambda) , \pi(-\lambda) \}$, we have  the identity
\[
\#  ( \pi^{-1}(y) \cap A ) + \# (\pi^{-1} (y) \cap -A) = p \text{ or } p-1 \, ,
\]
according to whether $p$ divides $m$  or not.

Suppose that the result holds true for $m$ where $m$ is an integer satisfying $\varphi(m) \ge 4$. We will
now show that the result also holds for $mp$ where $p$ is any odd prime.

Let $\mu=\pi(\lambda)$ and define $ B \subset U(m)$ such that for every $y \in B$ the cardinality of $\pi^{-1}(y) \cap A$
is at least $(p-1)/2$ and
\[
U(m) = B \coprod - B \coprod \{ \mu \} \coprod \{ - \mu \} \, .
\]

Suppose there exists $b_1, \ldots, b_N \in B$ such that $b_1 + \ldots + b_N = 0$. Choose $c_1, \ldots, c_N \in C(mp)$ such that
$c_i \in \pi^{-1}(b_i)$ and $c_1 + \ldots + c_N=0$. Set $X_i = ( \pi^{-1}(b_i) \cap A) - c_i$ and notice that by construction
$X_i \subset C(p)$. By Cauchy-Davenport Theorem (see for instance \cite[Lemma 2.14]{MR628618}) we have that
\[
\# ( X_1 + \ldots + X_N ) \ge \min \{ p, N(p-1)/2 - (N-1) \} \, .
\]
Therefore if $p \ge 3 + 4/(N-2)$ ( i.e. $N \ge 3$ and $p\ge 7$; or $N \ge 4$ and $p\ge 5$) then
\[
\# ( X_1 + \ldots + X_N ) =p \, .
\]
Under this condition,  we can take $x_i \in X_i$ summing up to zero. Since $x_i + c_i \in \pi^{-1}(b_i) \cap A$ and
$\sum (x_i + c_i) = 0$, we have proved the claim for $p\ge 5$.

If $p=3$ then either $\# ( \pi^{-1}(y) \cap A) \ge 2$ for every $y \in B$, or there exists $y\in B$ such that $\pi^{-1}(y) \cap A$
and $\pi^{-1}(-y) \cap A$ are both non-empty. In the first case for any $N \ge 2$, $\# (X_1 + \ldots + X_N)=3$ and therefore the $N$ that works for $B$ also works for $A$.
In the second case, for any $a_1 \in \pi^{-1}(y)\cap A$ and $a_2 \in \pi^{-1}(-y) \cap A$ we have that $a_1 + a_2$ belongs to $C(3)$. This shows that $N=6$ works
in the second case.

An adaptation of this argument  shows that $N=8$ works for numbers of the form $m=2^{k}$ for $k\ge 3$.
We already now that this holds true for $m = 2^3$.  Assuming that it works for $2^k$, consider the quotient map  $\pi:C(2^{k+1}) \to C(2^k)$.
Notice that $U(2^{k+1})$ is mapped to $U(2^k)$ and that $\pi(\lambda) \neq \pi(-\lambda)$.
Therefore any decomposition $U(2^{k+1}) = A \coprod -A \coprod \{ \lambda \} \coprod \{ -\lambda\}$
allows us to produce a decomposition of $U(2^k) = B \coprod -B$ such that the fibers
of $\pi$ over points of  $B$ intersect $A$. Theorem \ref{T:Speyer} tells us that we can choose points $b_i \in B$   such that
$b_1 + \ldots + b_4 =0$. As above choose $c_i \in \pi^{-1}(b_i)$ satisfying $c_1 + \ldots + c_4=0$. We consider $X_i = ( \pi^{-1}(b_i) \cap A) - c_i$ as before.
Of course,  $X_i \subset \ker(\pi)$ and consequently  $X_1 + \ldots  + X_4$ is formed by $2$-torsion points. Therefore $0  \in 2 (X_1 + \ldots + X_4)$ and
we can produce $8$ elements of $A$ summing up to zero.

It remains to check that the result holds for numbers of the form $mp$ with $\varphi(mp)>  4$, $\varphi(m)\le 2$ and $p\ge 3$.

For numbers of the form $mp$ with $\varphi(m)=2$ and $\varphi(pm) > 4$ we proceed as follows. Notice that $\varphi(m)=2$ implies $m \in \{ 3, 4, 6\}$.
If $p = 3$ then $\varphi(m)=2$  and $\varphi(pm) > 4$ implies $m \in \{ 3, 6 \}$ and $pm \in \{ 9,18 \}$. A direct study of these
two particular cases shows that $N=6$ works for them. So we can
further assume that $p \ge 5$.
Let $\pi: C(pm) \to C(m)$. The image $\pi(A)$ has cardinality $1$ or $2$. Therefore for some $y \in \pi(A) \subset C(m)$ we have that the cardinality of $\pi^{-1}(y) \cap A$ is
at least $(\varphi(pm) - 2 )/4  = \frac{p-2}{2}$. Since $p$ is odd and the cardinality is an integer,
we get that the cardinality of  $\pi^{-1}(y) \cap A$  is actually bounded from below by $\frac{p-1}{2}$.
Since $m y =0$ we can choose  $c_1, \ldots, c_m  \in \pi^{-1}(y)$ such that $c_1 + \ldots +c_m =0$. Consider $X_i = A \cap \pi^{-1}(y) -c_i \subset C(p)$.
Cauchy-Davenport implies that $0 \in X_1 + \ldots + X_m$ if $m \ge 4$, or $0 \in 2 X_1 + 2X_2  +2X_3$ if $m=3$.
As before, we conclude  that $N \in \{ 4, 6\}$ works in this case.

Finally, assume that $\varphi(m)=1$, i.e. $m=2$. In this case $\varphi(mp)>  4$ implies that $p\ge 7$.
Consider the map $\pi: C(2p)\to C(2)$.  Clearly, $\pi(A) = 1$ in this case. Let $c_1, \ldots, c_6 \in \pi^{-1}(1)$ be elements summing up to zero and set $X_i= A -c_i$. The sets $X_i$ have
each cardinality $(p-3)/2$, and Cauchy-Davenport tells us that  $0 \in X_1 + \ldots + X_6$. Thus  $N=6$ works in this case.
\end{proof}

\begin{cor}\label{C:Speyer}
Let $\F$ be a reduced foliation defined by a closed rational $1$-form $\omega$ on a projective surface $X$.
Assume that $\mathcal F$ does not admit a rational first integral. Let $D=(\omega)_{\infty}$ be the polar divisor of $\omega$.
If $\varphi: X \to X$ is an automorphism of $\mathcal F$ of finite order then, for some natural number $k\le 12$,   there exists
an element of $\beta \in H^0(X, \Sym^k \Omega^1_X(\log D))$ such that $\varphi^* \beta = \beta$ and
the contraction of $\beta$ with a twisted vector field
$v \in H^0(X,TX \otimes \KF)$ defining $\mathcal F$ gives rise to a section of $H^0(X,\KF^{\otimes k})$ vanishing
  along any irreducible $\F$-invariant invariant compact  curve contained in $X-D$.
\end{cor}
\begin{proof}
Let $C \subset X- D$ be the maximal $\F$-invariant compact curve contained in $X-D$.

By Corollary \ref{C:chave}, the vector space $V_C$ of logarithmic $1$-forms with poles on $D$ and with trivial restriction to $C$  is the image of $r^* : \h0A \to H^0(X,\Omega^1_X(\log D))$, for a certain  non-constant rational map $r: X\dasharrow A$ to a quasi-abelian variety $A$. In particular $V_C$ is nontrivial.

Since $\F$ does not admit a rational first integral, the pull-back $\varphi^* \omega$ is a complex multiple of $\omega$. In particular, $\varphi^*$ preserves the divisor $D$
and   also $C$, the maximal $\F$-invariant compact curve contained  in $X-D$.
Hence $V_C$ is also invariant under $\varphi^*$.  Let $\varphi_*$ be the automorphism of the
quasi-abelian variety $\Alb(X,D)= H^0(X, \Omega^1_X(\log D))^*/H_1(X-D,\mathbb Z)$ induced by $\varphi^*$.
Since  $\varphi^*$ preserves $V_C$, it follows that the action of $\varphi_*$ on $\Alb(X,D)$ preserves the image of
$\Alb(C) = \oplus \Alb(C_i)$ in $\Alb(X,D)$. We obtain in this way  an induced action of $\varphi_*$ on the quasi-abelian variety $A = \mathrm{coker}( \Alb(C) \to \Alb(X,D))$. This action
of $\varphi_*$ has the same eigenvalues as the action of $\varphi^*$ on $V_C$.
Therefore the action of $\varphi^*$ on  $V_C$   has eigenvalues described by  Proposition \ref{P:autA}. We can apply Theorem \ref{T:Speyer2}
in order to produce a symmetric differential  $\beta \in \Sym^k V_C$ for some $k \le 12$ which is invariant under $\varphi$, vanishes when restricted to $C$ and does not vanish identically along the leaves of $\F$. Thus the restriction of  $\beta$ to the leaves of $\mathcal F$ (i.e. contraction of $\beta$ with a twisted vector field defining $\F$) gives to a holomorphic section of $\KF^{\otimes k}$ with the sought properties.

In the case
$\omega \notin V_C$, we can apply Theorem \ref{T:Speyer}  instead of Theorem \ref{T:Speyer2} in order to get a section of $\KF^{\otimes k}$ for some $k\le 6$  with the required properties.
\end{proof}

\subsection{Synthesis}\label{S:Synthesis} The next result builds on the discussion carried out in this section. It goes a long way toward the
proof of Theorem \ref{THM:A}.

\begin{thm}\label{T:transEuclidean}
Let $\F$ be a reduced transversely affine foliation on a projective surface $X$ with transverse structure defined by a logarithmic  connection $\nabla$ on $N\mathcal F$ with finite monodromy. Assume that $\F$ does not admit a rational first integral then for some $k\le 12$ there exists a non-zero section $s$ of $\KF^{\otimes k}$.
Moreover,  the section $s$  vanishes on
all irreducible components of $(\nabla)_{\infty}$ with residue strictly greater than~$-1$.
\end{thm}
\begin{proof}
Let $p : Y \to X$ be the Galois morphism obtained through a minimal resolution of the ramified covering defined by the kernel of the monodromy of $\nabla$ as explained in \S \ref{S:finitegalois}.
Let us denote by $\varphi \in \Aut(Y)$ the automorphism of  $p$ and
by $\G$ the pull-back foliation $p^* \F$.
By assumption $\G$ is defined by a closed rational $1$-form $\omega$ with polar divisor $P$.
Let $C \subset Y$ be the maximal compact curve invariant by $\G$ and with no irreducible component contained in the support of  $P$. Lemma \ref{L:zerosandpoles} implies that $C$, if not empty, is a compact curve contained in $Y-|P|$. Notice that the pull-back to $Y$ of the irreducible
components of $(\nabla)_{\infty}$ with
residue strictly greater than $-1$ are contained in $C$.

Let $v\in H^0(Y,TY\otimes \KG)$ be a twisted vector field defining $\G$.
Corollary \ref{C:Speyer} produces a logarithmic symmetric differential $\sigma \in H^0(Y, \Sym^k \Omega^1_Y(\log P))$ invariant under $\varphi$ whose contraction with $v$ vanishes along $C$. Let $\sigma_{\vert T\G}\in H^0(Y,\KG ^{\otimes k})$ denote this contraction.
The  $\varphi$ invariance of $\sigma$
implies the $\varphi$ invariance of $\sigma_{|T\G}$. According to Lemma \ref{L:fix}  $\KG = p^* \KF$, therefore $\sigma_{|T\G}$ is the pull-back under $\pi$ of a section $s$ of $H^0(X, \KF^{\otimes k})$. As $\sigma_{|T\G}$ vanishes along $C$ and $\pi_{|C}$ is generically a submersion to $\pi(C)$, the section $s$ must vanish along all the irreducible components of $(\nabla)_{\infty}$ with
residue strictly greater than $-1$.
\end{proof}

\section{Foliations of Kodaira dimension one}

In this section we will investigate foliations of Kodaira dimension one (not necessarily transversely affine). The study carried out here will be relevant to the proofs of the three main theorems.
\subsection{Classification} The classification of foliations of Kodaira dimension one
 is due to Mendes  \cite{MR1785264}. The particular case of foliations of Kodaira dimension
 one admitting a rational first integral goes back to the work of Serrano.

\begin{thm}\label{T:kod1}
Let $\F$ be a reduced foliation on a smooth projective surface $X$. If $\kod(\F)= 1$ and $f: X \to C$ is the
pluricanonical fibration of $\F$ then either
\begin{enumerate}
\item the fibration $f$ is a non-isotrivial elliptic fibration and $\F$ is the foliation defined by $f$ or,
\item the fibration $f$ has rational fibers and $\F$ is a Riccati foliation relative to  $f$ or,
\item the fibration $f$ is an isotrivial elliptic fibration and $\F$ is a turbulent foliation  relative to $f$ or,
\item the fibration $f$ is an isotrivial hyperbolic fibration and $\F$
is also an  isotrivial hyperbolic fibration but $\F$ does not coincide with the foliation defined by $f$.
\end{enumerate}
\end{thm}

We will now obtain effective bounds on the least natural number $k$ such that $h^0(X , \KF^{\otimes k}) \ge 2$
analysing separately  each of the four cases predicted by Theorem~\ref{T:kod1}.

\subsection{Non-isotrivial elliptic fibrations}
The case of non-isotrivial elliptic fibrations is well-known. For instance it is implicitly treated in \cite{CasciniFloris}. Yet, for the
sake of completeness we state and prove the following result.

\begin{prop}
Let $\F$ be the reduced foliation subjacent to a non-isotrivial elliptic fibration $f:X \to C$ on
a projective surface $X$. Then $h^0(X, \KF^{\otimes 12})\ge 2$.
\end{prop}
\begin{proof}
The positive part of the Zariski decomposition of $\KF$ coincides with $f^* M_C$, where $M_C$ is  the moduli
part in the canonical bundle formula $K_{X/C} = f^*(M_C + B_C)$ of Kodaira, see \cite[Lemma 2.22]{CasciniFloris}.
Moreover   $|12M_C|$ is Cartier and base point free. More precisely, $12M_C$ coincides
with $J^* \mathcal O_{\mathbb P^1}(1)$ where $J:C \to \mathbb P^1$ is the $j$-invariant of the fibration according
 to \cite[Theorem 2.9]{MR816221}. Therefore $H^0(X, \KF^{\otimes 12}) = H^0(C, M_C^{\otimes 12}))$ and the result follows.
 \end{proof}

\subsection{Riccati foliations}\label{S:Riccati}
Let $\F$ be a  reduced Riccati foliation
on a projective surface $X$.
If $\F$ is not a fibration by rational curves then $\KF$ is pseudo-effective, in particular it admits a Zariski decomposition.
The structure of the negative part of the Zariski decomposition of $\F$ was precisely describe by McQuillan in \cite{MR2435846}:
in particular its connected components are formed by trees of rational curves. Furthermore, if one assumes that the foliation
$\F$ is relatively minimal then the connected components of the negative part of the Zariski decomposition are Hirzebruch-Jung
chains of $\F$-invariant rational curves. The contraction of the negative part gives raise to a singular projective surface with at worst
cyclic quotient singularities. The resulting foliation is called a nef model of $\F$. In the remainder of this section we will
work with foliations produced by this process. For details the reader can consult the original paper \cite{MR2435846} or the survey
\cite{MR2071237}.

Assume $\F$ is a nef model for a reduced Riccati foliation of non-negative Kodaira dimension. Let $f:X \to C$ be the reference (equivalently, adapted) fibration.
The canonical bundle of $\F$ is described in \cite[IV.4]{MR2435846}. We present an equivalent  description below
following Brunella's survey \cite[Section 7]{MR2071237}.  We start by recalling the classification of the fibers
of the reference fibration $f$, assuming $\KF$ is  nef. They are divided in $5$ classes labeled $(a)$, $(b)$, $(c)$, $(d)$ and $(e)$.
\begin{enumerate}
\item[$(a)$] Smooth fibers of $f$ transverse to $\F$.
\item[$(b)$] Singular fibers of $f$ transverse to $\F$  with two cyclic quotient singularities of the same order $o$.
\item[$(c)$] Smooth fibers of $f$ invariant by $\F$ with two non-degenerate saddles  or one saddle-node
with multiplicity two.
\item[$(d)$] Smooth fibers of $f$ invariant by $\F$ with two saddle-nodes with the same multiplicity $m$.
\item[$(e)$] Singular fibers of $f$ invariant by $\F$ with one saddle-node of multiplicity $l$ and two
quotient singularities of order $2$.
\end{enumerate}

The quotient singularities over fibers of  $f$ of type $(b)$ and $(e)$
induce a natural orbifold structure $\Corb$  on $C$ where the points below $(b)$ have multiplicity $o$
and the points below $(e)$ have multiplicity $2$. Therefore
\begin{equation}
K_{\Corb} = K_C + \sum_{(b)} \frac{o_j-1}{o_j}b_j+\sum_{(e)} \frac{e_j}{2} \,
\end{equation}
where the points $b_j,e_j \in C$ are below the  fibers of type $(b)$ and $(e)$ respectively.
Similarly, the direct image of the  canonical bundle of $\KF$ can be expressed as
\begin{equation}\label{E:KF}
  f_* \KF = K_{\Corb} + \sum_{(c)} c_j+\sum_{(d)} m_j d_j+\sum_{(e)} \frac{l_j}{2}e_j
\end{equation}
where $c_j,d_j$ and $e_j$ run respectively among the points of $C$ below fibers of type $(c)$, $(d)$ and $(e)$.  For a precise description of the coefficients $m_j$, $l_j$, $o_j$ the reader can consult Brunella's paper.  For our purposes, it is sufficient to know that $K\F=f^*(f_*K\F)$,  $o_j>1$ is equal to the finite order of the local monodromy around the fiber over $b_j$ and $m_j$, $l_j$ are positive integers.

Before investigation of the pluricanonical sections, we emphasize properties of  Riccati foliations and their pull-backs.

\subsubsection{Transversely affine Riccati foliations} In the following lemmas we collect properties of transversely affine Riccati foliations which will be
used in the remainder of this section and, more intensively, in Section \ref{S:algorithmic}. We recommend the reader to skip
this section in a first reading, and return to it when necessary.

\begin{lemma}\label{L:propR}
Let $\F$ be a transversely affine foliation on a projective surface $X$ which is not virtually transversely
Euclidean. The following assertions hold true.
\begin{enumerate}
\item \label{L:propR:1}
There exists a unique flat meromorphic connection $\nabla$ on $N\F$ which defines a transverse affine structure for $\mathcal F$.
\item  \label{L:propR:2}If the connection $\nabla$  is  logarithmic then its monodromy is an infinite subgroup of $\mathbb C^*$. In particular, if $X$ is simply connected then  the residues of $\nabla$ are not all rational.
\item  \label{L:propR:3} There exists a Ricatti foliation $\mathcal R$ on a ruled surface $f:S \to B$, and a rational map $p:X\dasharrow S$ such that
$\F$ is equal to $p^* \mathcal R$. The foliation $\mathcal R$ can be chosen to be reduced or, even more, can be chosen to be a nef model of a reduced foliation.
\item  \label{L:propR:4} The Riccati foliation $\mathcal R$ is transversely affine and not virtually transversely Euclidean.
 \item  \label{L:propR:5} The Riccati foliation $\mathcal R$ leaves invariant a unique algebraic section $\sigma$ of $f$. No other invariant algebraic curve dominates the basis of $f$. The connection on $N\mathcal R$ defining the affine structure for $\mathcal R$ has a simple pole along $\sigma$ with residue~$-2$.

 \item  \label{L:propR:6} If  $\calR$ is taken reduced and nef  then:
\begin{enumerate}[label=(\roman*)]
 \item  \label{L:propR:6:i} there are no invariant fibers of type $(e)$ (nilpotent fibers in the  terminology of \cite[Proposition~$4.2$]{MR3328860}).
 \item  \label{L:propR:6:ii} The components of $(\nabla)_{\infty}$ that map into fibers of type $(b)$ or $(c)$ of $f$ (nondegenerate fibers in \cite[Proposition~$4.2$]{MR3328860}) are simple poles of~$\nabla$.
\item  \label{L:propR:6:iii} If a fiber of $f\circ p$ contains a multiple pole of $\nabla$, then it is $\F$-invariant and contains a singularity possessing a  non algebraic (perhaps formal) separatrix for $\F$.  In particular this separatrix bears no pole of $\nabla$.

 \item  \label{L:propR:6:iv} The adapted ruling $f$ is the Iitaka fibration of $K\calR$, hence $\kod \calR=1$.
\end{enumerate}

\end{enumerate}
\end{lemma}

\begin{proof}
(\ref{L:propR:1}) The difference between two distinct flat connections on the same line-bundle is a non-trivial closed rational $1$-form $\omega$. If the two connections define transversely affine structures for the same foliation, then $\omega$ must vanish along the foliation. It follows that the foliation is transversely Euclidean, contrary to our assumption. (\ref{L:propR:2}) This follows from Remark~$\ref{rem:eucl}$. (\ref{L:propR:3}) This follows from Theorem~$\ref{T:structure}$ and existence of reduced nef models.
(\ref{L:propR:4}) It follows from \cite[Theorem 2.21]{MR2324555} that  $\calR$ is transversely affine. If $\calR$ would be virtually Euclidean, it would be the same for $\F$, by a fiber product argument, or Remark \ref{rem:eucl}.

(\ref{L:propR:5}) Let $\nabla_{\calR}$ be a connection on $N \mathcal R$ defining the transverse affine structure for $\calR$. If $F$ is a general fiber of the ruling $f$ then $N\mathcal R \cdot F = 2$ according to \cite[Chapter 4, Section 1]{MR3328860}. The residue formula (Proposition \ref{P:residues}) implies that $\supp(\nabla)_{\infty}$ intersects $F$. If an $\calR$-invariant algebraic curve intersects $F$ in more than one point then, after base change through a finite ramified covering of $B$ and birational trivialization of $f$, we can assume that $S=B\times \Pu$, $f$ is the first projection and, for $y$ a suitable  affine chart of $\Pu$,  the sections $y=0$ and $y=\infty$ are $\calR$-invariant. In particular $\calR$ is defined by a $1$-form
   $dy +  y \alpha $, where $\alpha$ is a rational $1$-form  on $B$.
It follows that $\mathcal R$ (after a ramified covering) is defined by the closed rational $1$-form $dy/y + \alpha$  and  is therefore virtually transversely Euclidean, contradicting (\ref{L:propR:4}). Thus,  there is a unique invariant algebraic curve  that  dominates the basis of the ruling and it is a genuine section $\sigma$ contained in the support of $(\nabla)_{\infty}$. Proposition~\ref{P:residues} implies that the residue of $\nabla$ along $\sigma$ equals  $-2$. Observing that $\mathcal R$ is defined in suitable affine coordinates by  a $1$-form
 \begin{equation}\label{Riccaff}dy + \beta + y \alpha, \end{equation}
 with $\alpha, \beta$ rational $1$-forms on $B$. One checks the local connection form for $\nabla_{\calR}$ is $f^*\alpha$ and  verifies the connection is logarithmic at the section $\sigma =\{ y=\infty \}$.

$(\ref{L:propR:6})-\ref{L:propR:6:i}$ It suffices to observe that  through the saddle-node of any fiber of type $(e)$, the unique (non-fiber) separatrix is a two-to-one covering of the basis. The existence part of $(\ref{L:propR:5})$ shows this separatrix must be $\calR$-invariant, contradicting the uniformity part of  statement of $(\ref{L:propR:5})$.

$(\ref{L:propR:6})-\ref{L:propR:6:ii}$ Let $f^{-1}(x)$ be a  fiber of type $(b)$ or $(c)$ and consider a polar component $C$ of $(\nabla)_{\infty}$ that maps in $f^{-1}(x)$ through $p$, \textit{i.e.} $f\circ p(C)=\{x\}$. After a birational $\Pu$-bundle transformation of $f:S\to B$, $f^{-1}(x)$ becomes a nondegenerate fiber as in \cite[Proposition~$4.2$]{MR3328860}). In particular,  one has an expression $(\ref{Riccaff})$ for $\calR$, with $\beta$ logarithmic at $x$.
At the neighborhood in $X$ of a smooth general point $z$ of $C$,  the map $p$ is holomorphic, and $\eta=p^*f^*\beta$ is logarithmic around $z$. Up to a birational line bundle transformation this latter form is the connection form for $\nabla$ around $z$. In particular, $\nabla$ has at most a simple pole at $C$.

$(\ref{L:propR:6})-\ref{L:propR:6:iii}$ If a fiber  $F$ of $f\circ p$ is not $\F$-invariant, the corresponding fiber of $f$ crosses infinitely many leaves of $\calR$. In particular, since $\calR$ is reduced and nef, this latter is of type $(a)$ or $(b)$. After $(\ref{L:propR:6})-\ref{L:propR:6:ii}$, this implies $F$ bears no multiple pole of $\nabla$.
So that if a fiber bears a multiple pole of $\nabla$, it is $\F$-invariant and, after $(\ref{L:propR:6})-\ref{L:propR:6:ii}$ and $(\ref{L:propR:6})-\ref{L:propR:6:i}$, it is of type $(d)$. Namely, this fiber contains two distinct saddle-nodes, with strong separatrix in the fiber. By (\ref{L:propR:5}), one of its two weak separatrices is not algebraic. This non-algebraic separatrix lifts to a non-algebraic separatrix for $\F$.

$(\ref{L:propR:6})-\ref{L:propR:6:iv}$
As, by formula (\ref{E:KF}), $K\calR$ comes from $B$, one must have $\kod \calR =-\infty, 0$ or $1$.  If we had $\kod \calR =-\infty$, by the classification of foliations, $\calR$ would be a fibration, in particular it would be transversely Euclidean, contradiction. Proposition \ref{P:kod0} excludes $\kod K\calR=0$. Hence $\kod K\calR=1$ and   $f$ is the Iitaka fibration of $K\calR$.
\end{proof}

\begin{lemma}\label{L:only dicritic}
Let $\calR$ be a transversely affine Riccati foliation on a rational surface. If a reduced nef model of $\calR$ does not have invariant fibers then $\calR$ is virtually transversely Euclidean.
\end{lemma}
\begin{proof}
Since the foliated surface is rational, so is the basis of the reference fibration.
Thus the monodromy group of $\calR$ is generated by the local monodromies around the points below the fibers of type $(b)$ of a nef model of $\calR$. As these local monodromies are all finite and share a common  finite orbit, corresponding to a horizontal components of the singular divisor of the transverse affine structure, the monodromy must be virtually abelian. Furthermore, due to the inexistence of invariant fibers we can apply \cite[Proposition 6.1]{LPT} to a reduced model  of $\calR$ to see that the connection on $N\tilde\calR$ is  logarithmic. The conclusion is given by Remark~\ref{rem:eucl}.  \end{proof}

\begin{lemma}\label{lemme fibration}
Let $\F$ be a foliation with  reduced singularities on a projective surface~$X$.
Assume $\F$ is a rational pull-back of a Riccati foliation $\calR$ on a surface and $\F$ is not virtually transversely Euclidean. The  pull-back $f:X\dasharrow C$ of the adapted fibration of $\calR$ is still a fibration.
\end{lemma}
\begin{proof}
If $\calR$ is not transversely affine, then it has no invariant section and any indeterminacy point of $f$ is a dicritical singularity of $\G$, contradiction.

One may thus assume that $\F$ is transversely affine, not virtually transversely Euclidean. There is a unique reference
fibration for $\calR$ according to Lemma \ref{L:propR}, item $(\ref{L:propR:6})-\ref{L:propR:6:iv}$.

Assume first that the connection $\nabla$ on $N\F$  defining the transverse affine structure of
$\F$ is not logarithmic. On the one hand, the irregular divisor
$D=(\nabla)_{\infty} - ((\nabla)_{\infty})_{\text{red}}$  has self-intersection zero according to \cite[Proposition 6.1]{LPT}. On the other hand, the connection on $N\F$ is induced by the pull-back of the corresponding connection for $N\calR$, the normal bundle of the Riccati foliation $\calR$.
Consequently, the irregular divisor is  a finite union of (multiples of) fibers of the pull-back of the reference fibration. We deduce from $D^2=0$ that the pull-back of fibers of the reference
fibration are disjoint as wanted.

It remains to deal with the case where $\nabla$ is logarithmic. Aiming at a contradiction, assume that the
pull-back of the reference fibration has a base point. Resolution of indeterminacies (i.e. resolution of the base points)  yields an irreducible component $E$ of the  exceptional  divisor  which dominates the basis of the fibration. If $\tilde \F$ is the strict transform of $\F$ then $E$ must be $\tilde \F$-invariant as otherwise $\F$ would have a dicritical singularity. Since dicritical singularities are non-reduced, this contradicts our assumptions. For the same reason, the holonomy of $\tilde \F$ along $E- \sing(\tilde \F)$ must be abelian as exceptional divisors over reduced singularities have abelian holonomy.  Since we know that the monodromy representation  of the transverse affine structure for $\tilde \F$  factors through the basis of the fibration, we deduce from the commutativity of the holonomy of $\tilde \F$ along $E$  that the monodromy of the transverse affine structure of $\tilde \F$ is virtually abelian. Therefore, by Remark~\ref{rem:eucl}, $\tilde \F$ is virtually transversely Euclidean and we obtain the sought contradiction.
\end{proof}

\subsubsection{Bounds}

\begin{prop}
Let $\F$ be a reduced and nef Riccati foliation
on a projective surface $X$, with $\kod(\F)=1$. Then $h^0(X, \KF^{\otimes k}) \ge 2$ for some $k \le 42$.
Moreover, if $\F$ is a transversely affine Riccati foliation that is not virtually transversely Euclidean, then we obtain the sharper bound $k\le 6$.
\end{prop}

The proof of this proposition relies on an analysis of the formula (\ref{E:KF}) for $f_* \KF$ presented in Section \ref{S:Riccati}.
We will split it in four lemmas. The first three determine $k$  according to the genus of the base of the reference fibration $f:X \to C$.
The last one deals with transversely affine Riccati foliations.
The starting point of all the three lemmas is the same. For every $k \in \mathbb N$, we have an
inclusion
\[
   H^0(C , \lfloor (f_* \KF)^{\otimes k} \rfloor) \xrightarrow{f^*} H^0(X, \KF^{\otimes k}) \, ,
\]
where $\lfloor (f_* \KF)^{\otimes k}) \rfloor$ is the divisor deduced from Equation (\ref{E:KF}) by rounding down to integers the rational coefficients.

\begin{lemma}\label{E:HH}
If $C$ has genus at least two then $h^0(X, \KF) \ge 2$.
\end{lemma}
\begin{proof}
It suffices to notice that the natural composition
$f^* \Omega^1_C \to \Omega^1_X \to \KF$ induces an injection of $H^0(C, \Omega^1_C)$  into $H^0(X,\KF)$.
\end{proof}

\begin{lemma}\label{L:HE}
If $C$ has genus one then $h^0(X, \KF^{\otimes k})\ge 2$ for some $k\le 4$.
\end{lemma}
\begin{proof}
Since $\F$ has Kodaira dimension one, we see that $\deg f_* \KF>0$. We claim that it
suffices to determine $k\ge 1$ such that $\deg \lfloor (f_* \KF)^{\otimes k}\rfloor \ge 2$.
Indeed if $L$ is a line bundle over a curve $C$ of genus $1$ satisfying $\deg(L)>0$ then Riemann-Roch
implies that $h^0(C,L) = \deg L+h^0(C,K_C\otimes L^*) = \deg(L)$. Since
all the coefficients in Equation (\ref{E:KF}) are at least $1/2$ and $K_C$ does not contribute to the computation
of the degree, it is clear that $\deg \lfloor (f_* \KF)^{\otimes k}\rfloor \ge 2$ for some $k\le 4$.
\end{proof}

\begin{lemma}\label{L:HP1}
If $C$ is a rational curve then $h^0(X, \KF^{\otimes k})\ge 2$ for some $k \le 42$.
\end{lemma}
\begin{proof}
If we set  $\delta = \deg(f_* \KF)$ and $\delta_k = \deg( \lfloor (f_*\KF)^{\otimes k}\rfloor )$ then
from Equation (\ref{E:KF}) we deduce that
\[
   \delta_k=-2k+\sum_{(b)}\lfloor k\frac{o_j-1}{o_j}\rfloor+\sigma_k,
\]
with
\[
   \sigma_k=\sum_{(c)}k+\sum_{(d)}km_j+\sum_{(e)}\lfloor k\frac{l_j+1}{2}\rfloor.
\]
We proceed according to the value of $\sigma= \sum_{(c)} 1+\sum_{(d)} m_j+\sum_{(e)}  \frac{l_j+1}{2} \in \mathbb Q$.
We will denote by  $k_{min}$ the smallest $k$ such that $\delta_k\ge 1$.

If $\sigma>2$ then automatically $\sigma\geq \frac{5}{2}$. In particular if $k>0$ is even
then $\sigma_k=k\sigma\geq\frac{5k}{2}$, $-2k+\sigma_k\geq 1$, $\delta_k\geq 1$ and we conclude $k_{min}\leq 2$ when $\sigma>2$.
		
If $\sigma\leq 2$ then $\sigma \in \{2, 3/2,1,0\}$. We will now analyze each one of this possibilities.

If $\sigma=2$ then, for $k$ positive and  even, $-2k+\sigma_k=0$ and $\delta_k\geq 1$ amounts to $\sum_{(b)}\lfloor k\frac{o_j-1}{o_j}\rfloor\geq 1$.
This   sum possesses at least one term, as $\delta=deg(f_*\KF)$ is positive. This term is $\lfloor k\frac{o_j-1}{o_j}\rfloor\geq\lfloor k/2\rfloor\geq 1$.
We thus also have $k_{min} \leq 2$ when $\sigma=2$.
			
If $\sigma=\frac{3}{2}$ then, for $k>0$ even, $\delta_k=-\frac{k}{2}+ \sum_{(b)}\lfloor k\frac{o_j-1}{o_j}\rfloor\geq -\frac{k}{2}+ \sum_{(b)}\frac{k}{2}$.
If there are two fibers of type $(b)$, we obtain $\delta_k\geq 1$ and $k\leq 2$.
If there are less than two fibers of type $(b)$ then $\delta>0$ imposes $-\frac{k}{2}+\sum_{(b)}\lfloor k\frac{o_j-1}{o_j}\rfloor\geq -\frac{k}{2}+\lfloor\frac{2k}{3}\rfloor$ and $k=6$ yields
$\delta_k\geq1$. Hence  $k_{min}\leq 6$ for $\sigma=3/2$.

If $\sigma=1$, $\delta>0$ forces $\sum_{(b)}\frac{o_j-1}{o_j}>1$, and if $k=6l, l>0$, then $\delta_k\geq -k+\frac{k}{2}+\frac{2k}{3}$ and $k=6$ yields $\delta_k\geq 1$.
Therefore  $k_{min}\leq 6$ when $\sigma=1$.

It remains to analyze the possibilities for $\sigma=0$. In this case   there are only fibers of type $(a)$ and $(b)$.
The condition $\delta>0$ means that $\Corb$  is a hyperbolic orbifold.
It is well known this forces the presence of at least $3$ fibers of type $(b)$.

If there are at least four such fibers, and $k=6l, l>0$, then $\delta>0$ implies $\delta_k\geq -2k+\frac{3k}{2}+\frac{2k}{3}$ and $k=6$ yields $\delta_k\geq 1$, hence $k_{min}\leq 6$.

Suppose there are exactly three fibers of type $(b)$. We consider $k_{min}$, $\delta$ and $\delta_k=\delta_k(o_1,o_2,o_3)$ as functions of the three weights $(o_1,o_2,o_3)$.
We will always suppose the  triples satisfy $o_1\leq o_2\leq o_3$.
We will write $(l_1,l_2,l_3)\geq(o_1,o_2,o_3)$  when   $l_j\geq o_j$ for all $j\in \{1,2,3\}$.
The key observation  is that $(l_1,l_2,l_3)\geq(o_1,o_2,o_3)$ implies $\delta_k(l_1,l_2,l_3) \geq \delta_k(o_1,o_2,o_3)$ and thus $k_{min}(l_1,l_2,l_3) \leq k_{min}(o_1,o_2,o_3)$.
						
If $o_1\geq 3$ then $\delta(o_1,o_2,o_3)>0$ imposes $(o_1,o_2,o_3)\geq(3,3,4)$. Direct computation shows $k_{min}(3,3,4)=12$, hence $o_1\geq 3$ implies $k_{min}(o_1,o_2,o_3)\leq12$.

If $o_1\leq 2$, then $\delta(o_1,o_2,o_3)>0$ imposes $(o_1,o_2,o_3)\geq(2,3,7)$ or $(o_1,o_2,o_3)\geq(2,4,5)$. Direct computation shows $k_{min}(2,3,7)=42$  and $k_{min}(2,4,5)=20$. Thus in any case $k_{min} \le 42$.
\end{proof}

\begin{remark}\label{non dicritical}
Pushing further the analysis above one can obtain the following more precise picture.
If there is at least one fiber which is not of  type $(a)$ or $(b)$ then  $k_{min}\leq6.$
If there are only fibers of type $(a)$ and $(b)$ then $\Corb$ is a hyperbolic orbifold and as such contains at least  $3$ non-smooth  points.
		\begin{enumerate}[label=$\roman*$)]
			\item if $\Corb$ contains strictly more than $3$ non-smooth  points then $k_{min}\leq6$.
			
			\item if $\Corb$ contains exactly $3$ points then
$k_{min}=42$  for $(2,3,7)$;
$k_{min}=24$ for $(2,3,8)$;
$k_{min}=20$ for $(2,4,5)$;
$k_{min}=18$ for $(2,3,9), (2,3,10)$ and $(2,3,11)$; and  $k_{min}\leq 12$ for any other case.
		\end{enumerate}
\end{remark}

\begin{lemma}\label{L:4.8}
Let $\calR$ be a reduced and nef Riccati  foliation on a  rational surface.
Assume that $\calR$ is transversely affine and not virtually transversely Euclidean.
Then $h^0(X, K\calR^{\otimes k}) \ge~2$ for some $k \le 6$.
\end{lemma}
\begin{proof}
By Lemma~\ref{L:propR} $(\ref{L:propR:6})-\ref{L:propR:6:iv}$, $\kod \calR=1$ and by Lemma~\ref{L:only dicritic}, there cannot be only fibers of type $(a)$ and $(b)$, Remark~\ref{non dicritical} yields the conclusion.
\end{proof}

\subsection{Turbulent foliations}
The picture in the case of turbulent foliations is very similar to the one drawn above for Riccati foliations, see \cite[IV.4]{MR2435846} and  \cite[Section 7]{MR2071237}.
There exists a nef model and a precise description of the fibers of the reference fibration $f:X \to C$ very similar to the
one for Riccati foliations. A key difference is that the multiplicity of fibers of $f$, unlikely in the case of Riccati foliations,
are uniformly bounded. The only possibilities are $\{1,2,3,4,6\}$ (possible orders of automorphisms of an elliptic curve fixing a point).

\begin{prop}
Let $\F$ be a  turbulent  foliation of Kodaira dimension one on a projective surface.
Then $h^0(X, \KF^{\otimes k}) \ge 2$ for some $k \le 12$.
\end{prop}
\begin{proof}
Since the multiplicities of multiple fibers belong to $\{1,2,3,4,6\}$ , we deduce that  $(f_* \KF)^{\otimes 12}$ is an effective $\mathbb Z$-divisor of positive degree.  Consequently, we must
have $\deg((f_* \KF)^{\otimes 12}) \ge 1$. As before if the base $C$ is hyperbolic then we already have that $h^0(X,\KF)\ge 2$.
If it is rational then we also have $h^0(X,\KF^{\otimes 12})\ge 2$.
If instead the base is an elliptic curve then either some $k\le 6$ suffices to ensure that $(f_* \KF)^{\otimes k}$ is a $\mathbb Z$-divisor of positive
degree or there at least two distinct orbifold points on $C$. In any case we guarantee that  $h^0(X,\KF^{\otimes 12})\ge 2$  arguing as in Lemma \ref{L:HE}.
\end{proof}

\subsection{Isotrivial hyperbolic fibrations} It remains to analyze the case of
isotrivial hyperbolic fibrations.

\begin{prop}
Let $\F$ be an isotrivial hyperbolic fibration.
Then $h^0(X, \KF^{\otimes k}) \ge 2$ for some $k \le 42$.
\end{prop}
\begin{proof}
As in the case of Riccati foliations there is no uniform bound for the multiplicity of the fibers.
Nevertheless there exists a nef model, and all the fibers of the   pluricanonical fibration $f:X\to C$  on the nef model
are transverse to the foliation. The multiple fibers induce an orbifold structure on $C$ and $f_* \KF = K_{\Corb}$.
The analysis carried out above for Riccati foliations shows that  $h^0(C, K_{\Corb}^{\otimes k}) \ge 2$ for some $k\le 42$.
\end{proof}

\subsection{Synthesis} We collect  the results spread over the four previous subsections in the statement below.

\begin{thm}\label{T:Kod1}
Let $\F$ be a reduced foliation of Kodaira dimension one on a projective surface $X$.
Then $h^0(X, \KF^{\otimes k})\ge 2$ for some $k \le 42$. Moreover, if $\F$ is
a transversely affine Riccati foliation which is not virtually transversely Euclidean then we can take $k \le 6$.
\end{thm}

\section{Existence of invariant curves of small degree}
The following will be used to prove Theorems \ref{THM:A} and \ref{THM:A2}.

\begin{lemma}\label{L:sections}
Let $\F$ be a reduced foliation on a projective surface $X$.
Assume that $\F$ is the pull-back under a morphism  $\pi : X \to Y$ of a reduced foliation $\G$
on another projective surface $Y$. Then, for  $k>0$, any non trivial section $s$ of $\KG^{\otimes k}$ induces a section $\pi^* s$  of $\KF^{\otimes k}$.
 Furthermore the zero locus of $\pi^* s$ contains the inverse image of the zero locus of $s$.
\end{lemma}
\begin{proof}
Since $\G$ is reduced, it follows that $\KF - \pi^* \KG = \Delta$ where $\Delta$ is an  effective divisor, see \cite[Proposition 2.1]{MR2818727}.
Therefore $H^0(X , \pi^* \KG^{\otimes k})$ injects into $H^0(X ,  \KF^{\otimes k})$, for every $k>0$, in such way that a section $s$ is mapped
to a section $\pi^*s$ with zero divisor equal to the sum of the pull-back of $(s)_0$ with $k \Delta$. This is sufficient to prove the lemma.
\end{proof}

\subsection{Proof of Theorem \ref{THM:A}} Let $\F$ be a transversely affine foliation of degree $d\ge 2$ on $\mathbb P^2$ with transverse
structure given by a meromorphic flat connection  $\nabla$ on $\NF$.
Recall that in this case $\NF= \mathcal O_{\mathbb P^2}(d+2)$ and $\KF = \mathcal O_{\mathbb P^2}(d-1)$.
Assume also that $\F$ does not admit a rational first integral.

Let us first treat the case where $\nabla$ is logarithmic with finite monodromy.
Thus all the residues of $\nabla$ are rational numbers. According
to Proposition \ref{P:residues} we have the identity
\[
 d+2 = - \sum_{C \subset \mathbb P^2} \Res_C(\nabla) \deg (C) \, .
\]
If all the residues of $\nabla$ are $\le -1$ then it follows that the support of $(\nabla)_{\infty}$ has degree bounded by $d+2$ which is
certainly smaller than $12(d-1)$ since $d \ge 2$.

Assume that there are residues which are strictly greater than $-1$. Let $\pi: X \to  \mathbb P^2$ be a resolution of singularities
of $\F$ and let $\G = \pi^* \F$. Theorem \ref{T:transEuclidean} guarantees the existence of a section $s$ of $\KG^{\otimes k}$ for some
$k \le 12$ which vanishes along all the irreducible components of $(\nabla_{\G})_{\infty}$ with residues strictly greater than $-1$. This
section descends to a section of $\KF^{\otimes k}$ vanishing along the irreducible components of $(\nabla)_{\infty}$ having the same properties.
This suffices to prove the  Theorem  in the case $\nabla$ is logarithmic with finite monodromy.

If $\F$ admits more than one transverse affine structure
then we can always choose $\nabla$ logarithmic and with trivial monodromy, cf. \cite[Proposition 2.1]{MR3294560}.

It remains to deal with the case where $\F$ admits a unique transverse affine structure defined by a meromorphic flat
connection $\nabla$ which is not logarithmic or has infinite monodromy. The structure theorem for transversely affine foliations tells us
that there exists a ruled surface $S$, a Riccati foliation $\mathcal R$ on $S$, and a rational map $p: \mathbb P^2 \dashrightarrow S$ such that  $\F = p^* \mathcal R$.
Since we are assuming that $\F$ is not virtually transversely Euclidean, Proposition~\ref{P:kod0}  implies that $\mathcal R$ has Kodaira dimension one.

Let $f : \mathbb P^2 \dashrightarrow \mathbb P^1$ be a rational map with irreducible general fiber defining  the
pull-back to $\mathbb P^2$ under $p$ of the reference fibration of $\mathcal R$. If
none of the irreducible components of $(\nabla)_{\infty}$ is contained in fibers of $f$ then they all come from horizontal algebraic leaves of
the Riccati folation $\mathcal R$. Since the horizontal leaves of a transversely affine Riccati foliation have integral residues $\le -1$, then
all the irreducible components of $(\nabla)_{\infty}$ have integral residues $\le -1$.
Therefore, as in the case where $\nabla$ is  logarithmic  with finite monodromy we can conclude that the support of $(\nabla)_{\infty}$
has degree $\le d+2  <6(d-1)$.

Assume from now on that  some irreducible component of $(\nabla)_{\infty}$ is contained in a fiber of $f : \mathbb P^2 \dashrightarrow \mathbb P^1$. To conclude the proof of the theorem it suffices to bound the degree of fibers of $f$. For that replace $\mathcal R$ by a reduced foliation
birationally equivalent to it and resolve the  rational map $p:\mathbb P^2 \dashrightarrow S$ to obtain the commutative diagram below.
\begin{center}
\begin{tikzpicture}
  \matrix (m) [matrix of math nodes,row sep=3em,column sep=4em,minimum width=2em]
  {
    X & \, \\
    \mathbb P^2 & S \\};
  \path[-stealth]
    (m-1-1) edge node [above] {$\tilde p$} (m-2-2)
    (m-1-1) edge node [left] {$q$} (m-2-1);
      \path[-stealth]
  [dashed]        (m-2-1) edge node [below] {$p$} (m-2-2)   ;
\end{tikzpicture}
\end{center}
Let $\G = q^* \F$ be the pull-back of $\F$ to $X$ and notice that it coincides with $\tilde p^* \mathcal R$.
We know from Lemma \ref{L:4.8} that  there exists two linearly independent sections of $\KR^{\otimes k}$, say $s_1, s_2$, for some $k\le 6$.
The rational map to $g: S \dashrightarrow \mathbb P^1$ defined by the quotient of these two rational sections contracts  the fibers of the reference fibration
of $\mathcal R$. In other words, its Stein factorization coincides with the reference fibration of $\mathcal R$. Therefore
$\tilde{p}^*s_1, \tilde{p}^*s_2 \in H^0(X, \tilde p^*\KR^{\otimes k})$, the pull-backs of $s_1, s_2$  to $X$,
define a rational map $X \dashrightarrow \mathbb P^1$
which contracts the fibers of the pull-back of the reference fibration. In particular, there exists a section in $H^0(X, \tilde p^*\KR^{\otimes k})$
which vanishes along any given fiber of the pull-back of the reference fibration.
We can apply Lemma \ref{L:sections} to obtain a section $s \in H^0(X, \KG^{\otimes k})$ ($k\le 6$) with the very same property.
The push-forward of $s$ to $\mathbb P^2$ is a section of $\KF^{\otimes k} \simeq \mathcal O_{\mathbb P^2} (k(d-1))$ which vanishes
on any given fiber of $f$. Thus the fibers of $f$ have degree bounded by $6(d-1)$. This concludes the proof of Theorem \ref{THM:A}. \qed

\subsection{Proof of Theorem \ref{THM:A2}}
Let $\F$ be a foliation of degree $d$ on the projective plane. If $d\le 1$ and $\F$ admits a rational first integral then $\F$ is a pencil of rational curves
and  does not satisfy our assumptions. Thus $d\ge 2$.  Let $\pi:X \to \mathbb P^2$ be a reduction
of singularities of $\mathcal F$. Denote by $\G = \pi^*\F$ the resulting reduced foliation. Notice that $\G$
is not necessarily relatively minimal, but since we are assuming that the general leaf of $\F$ has genus $>0$,
we have that $\KG$ is pseudo-effective according to Miyaoka's Theorem \cite[Theorem 7.1]{MR3328860}. If we consider the Zariski decomposition $P+N$ of $\KG$ then the support
of $N$ is $\G$-invariant, see \cite[Proposition III.2.1]{MR2435846} or \cite[Theorem 8.1]{MR3328860}.

If $\G$ is a foliation of Kodaira dimension zero then for
any $k>0$ and any non-zero section $s$ of $\KG^{\otimes k}$, the zero locus of $(s)$ is contained in the support of $N$.
Thus it must be  $\G$-invariant.  According to \cite{MR2177633}, there exists
a non-zero section of  $\KG^{\otimes k}$  for some $k \le 12$. Moreover, if we assume that the leaves of  $\G$ are algebraic then we
can take $k \le 6$.
Since $\pi_* \KG^{\otimes k}$ injects in $\KF^{\otimes k}$ with cokernel supported in codimension two, and $\KF= \mathcal O_{\mathbb P^2}(d-1)$ it follows
the existence of an invariant algebraic curve for $\F$ of degree at most $6(d-1)$

If $\G$ is a non-isotrivial elliptic fibration then the general member of  the linear system $|\KG^{\otimes 12}|$ is supported on a union
of fibers of the fibration. It follows that every leaf of $\F$ has degree bounded by $12(d-1)$.

Assume now that $\G$ is an isotrivial hyperbolic fibration, i.e. the general leaf has genus at least $2$. Thus the Kodaira dimension of $\G$ is one and its  Iitaka fibration has
 general fiber completely transverse to $\G$. Hence the Iitaka fibration defines  a foliation $\H$ such that $\tang(\G,\H)$ is invariant by both foliations, since otherwise this
 tangency locus would intersect a general fiber of both fibrations.
Since $h^0(X,\KG^{\otimes k}) \ge 2$ for some $k\le 42$, the direct image of $\H$ on $\mathbb P^2$  is a foliation defined by a pencil of curves of degree bounded by $42(d-1)$. The tangency locus of $\pi_* \H$  and $\F$ is invariant by both foliations. It follows that $\F$ has an invariant algebraic curve of degree at most $42(d-1)$. Theorem \ref{THM:A2} is proved. \qed

\section{Algorithmic  Liouvillian integration}\label{S:algorithmic}

In this last section, we present an algorithm which computes the transverse affine structure $\nabla :\NF \longrightarrow  \NF \otimes \Omega^1_X(*D)$ when it exists. Meanwhile, we will obtain four estimates (\ref{bound:I})-(\ref{bound:IV}) that prove Theorem~\ref{THM:C}.

We will not dwell upon the actual implementation of our algorithm.  In particular, we will assume that we can compute the reduction of singularities of any given foliation (a procedure not free from subtle computational pitfalls as discussed in \cite{MR2566144}), and that we can determine all the algebraic curves of degree smaller than a given integer which are invariant by a given foliation of $\mathbb P^2$(see \cite{MR1860669} and \cite{MR3454369} for an algorithm to carry out this task).

The main algorithm is decomposed in two parts. The first finds an integrating factor for virtually transversely Euclidean foliations (case~(\ref{struct1})  of Theorem~\ref{T:structure}).
The second part finds the integrating factor if it exists and we are \emph{not} in the above mentioned case~(\ref{struct1}).  These are described in Sections \ref{algo closed} and \ref{algo Riccati}, respectively. Sections \ref{basic estimates} and \ref{poles and residues} contain preparatory lemmas to be used in both cases. Both parts of the algorithm search the integrating factor amongst finitely many possibilities so that, in case no integrating factor exists, the algorithm still terminates in finite time.

The case of foliations of degree zero or one is very simple. They are all either algebraically integrable or transversely Euclidean (defined by a closed rational $1$-form). We will assume, without loss of generality,  that the degree of the foliation under study is at least two.

\subsection{Basic estimates}\label{basic estimates}
We present here bounds to be used in both parts of the proof in order to get the existence of the announced constant $\cnst$.

Recall the Milnor number (or multiplicity) of a singularity $p$ of a foliation $\G$ defined around $p$ by a local holomorphic $1$-form  $\omega=a(x,y)dx+b(x,y)dy$ with isolated zero is $\mu(\G,p)=dim_{\C} \C\{x,y\}/(a,b)$.
If a reduced invariant analytic curve $C$ has local equation $f=0$ then, after \cite[Chapter V]{MR1649358}, there exists $g,h\in \C\{x,y\}$, both prime to $f$, and $\eta$ a local holomorphic $1$-form such that $g\omega=hdf+f\eta$. The Camacho-Sad index $CS(\G,C,p)$ is then defined as $\Res_p(-(\eta/h)_{\vert C})$. The index $Z(\G,C,p)$ is defined as the (possibly negative) vanishing order at $p$ of $(h/g)_{\vert C}$. For a compact curve $C$, one defines $Z(\G,C)$ as the sum $\sum_{p\in C}Z(\G,C,p)$.
\begin{lemma} \label{lem:bound}
Let $(X,\G)$ be a reduction of singularities of a degree $d$ foliation $\F$ of $\P^2$.  At each singularities of $\G$, we have one or two local irreducible convergent separatrices, concatenate all the corresponding  (one or two terms) sequences of Camacho-Sad indices in a sequence $(i_{\ell})_{1\leq \ell \leq r}$. Let $C$ and $D$ be $\G$-invariant curve with no common component and let $p\in \sing(\G)$. Set $n=\card \sing(\G)$ and
\[\Gamma_1=\max_I \sum_{\ell\in I} i_{\ell},~~
\Gamma_2=\min_I \sum_{\ell\in I} i_{\ell},~~
\Gamma_3=d^2+d+1,\] where the maximum and minimum are taken among all subsets $I\subset \{1,\ldots, r\}$ such that the considered sum is an integer.
Set also $\Gamma_4=\frac{3}{2}\max(\Gamma_1,-\Gamma_2)+2n$.
Then

\begin{align}
 \label{lem:bound:i} \Gamma_2  &\leq  C \cdot C  \leq  \Gamma_1 +2n, \tag{$i$}\\
 \label{lem:bound:ii} 0  &\leq \mu(\G,p) \leq  \Gamma_3\tag{$ii$}\\
 \label{lem:bound:iibis}0  &\leq Z(\G,C) \leq  n\Gamma_3\tag{$iii$}\\
 \label{lem:bound:iii} \Gamma_2 &\leq  N\G\cdot C \leq n(\Gamma_3+2)+\Gamma_1,\tag{$iv$}\\
  \label{lem:bound:iv}  \Gamma_2/2-\Gamma_1-2n &\leq C\cdot D \leq \Gamma_1/2-\Gamma_2+n\tag{$v$}.
\end{align}
In particular, we have $\vert C\cdot C\vert \leq \Gamma_4,   \vert C\cdot D\vert \leq \Gamma_4, \vert N\G\cdot C\vert \leq n\Gamma_3+\Gamma_4$.

\end{lemma}
\begin{proof}
By Camacho-Sad formula \cite[Theorem $3.2$]{MR3328860}, $C\cdot C$ is the sum of the local indices $CS(\G,C,p)$, where $p$ runs among the singularities of $\G$ contained in $C$. Each of these local indices is either one of the $(i_{\ell})$ corresponding to $p$, or there are two such terms $i_{\ell}, i_{\ell+1}$ and  $CS(\G,C,p)=i_{\ell}+ i_{\ell+1}+ 2$, \cite[Proposition $3.1$ p. $157$]{MR1649358}. This yields (\ref{lem:bound:i}).
 From the proof of Seidenberg's theorem \cite[Appendice I]{MR608290},\cite[Theorem~$1.1$]{MR3328860}, $\mu(\G,p)\leq\mu(\F,q)$ where $p$ maps to $q\in \P^2$. Moreover, after \cite[Proposition~$2.1$]{MR3328860}, $\mu(\F,q) \leq d^2+d+1$, whence  (\ref{lem:bound:ii}).
 As $\G$ is reduced, calculation shows, for every singularity $p$,  $ 0 \leq Z(\G,C,p)\leq \mu(\G,p)$. Summing up (\ref{lem:bound:ii}) over all singularities thus yields (\ref{lem:bound:iibis}). From \cite[Proposition $2.3$]{MR3328860}, we have $N\G\cdot C=C\cdot C+ Z(\G,C)$, this allows to infer (\ref{lem:bound:iii}) from (\ref{lem:bound:i}) and (\ref{lem:bound:iibis}).
Estimate (\ref{lem:bound:iv}) follows from $2 C\cdot D= (C+D)^2-C^2-D^2$ and (\ref{lem:bound:i}).
\end{proof}

With the notation of Lemma~\ref{lem:bound}, the bound $\cnst$ we shall obtain for Theorem~\ref{THM:C} will be a function of the five constants $n,d,\Gamma_3$, $\Gamma_4$, $N\G^2$ mentioned above and one extra constant. The extra constant, called $\Lambda$, is defined  by the formula
\[
    \Lambda=\max\left(  \mathbb Q \cap  \bigcup_{\ell=1}^r \{-i_{\ell}\} \right) ,
\]
unless the set inside the parenthesis is empty, in which case we set $\Lambda =0$.

\subsection{Poles and residues of transverse affine structures}\label{poles and residues}

\begin{lemma} \label{chains}
Let $\G$ be a reduced foliation on a smooth complex surface $S$ which is transversely affine with connection $\nabla$. Let $E_1$,$E_2$ be two germs of invariant irreducible formal curves at $p\in S$. Suppose that $\G$ is defined by a a germ of vector field at $p$ with eigenvalue $1$ in the direction $T_pE_1\subset T_pX$  and eigenvalue $\lambda\in \C\setminus \Q_{>0}$ in the direction  $T_pE_2$ .
Denote $m_1,m_2$ the order of the poles of $\nabla$ at $E_1$ and $E_2$ respectively.
\begin{enumerate}
\item\label{lin} If $\lambda\in \C\setminus \Q$  then $m_1\leq 1$,  $m_2\leq 1$.
\item If $\lambda\in \Q^*$ then  either both poles  have order $1$ or both poles have order $\geq 2$. In any case, $m_2-1=-\lambda (m_1-1)$.
\item \label{saddle-node} If $\lambda=0$ then $m_2\leq1$, and $m_1\in\{1,\mu(\G,p)\}$.
Moreover, if $m_1,m_2\leq 1$, the residues of $\nabla$ at  $E_1$ and $E_2$ belong to $\R_{\leq-1}$.
\end{enumerate}
In any case,  a bound for  $m_1$ induces a bound for $m_2$ and \textit{vice-versa}.
\end{lemma}
\begin{proof}
Let $(\omega,\eta)$ a local pair defining $(\F,\nabla)$ around $p$.
Suppose there exists a formal diffeomorphism $\phi: (\C^2,0 )\to (X,p)$ such that $\phi^*\omega=ug\tilde{\omega}$ where $u$ is a formal unit, $g$ is a convergent meromorphic function and $\tilde{\omega}$ is a meromorphic closed $1$-form. Assume also one separatrix of $\F_{\tilde{\omega}}$ has infinite holonomy.  Then the formal meromorphic $1$-form $\tilde{\eta}=\phi^*\eta+\frac{du}{u}+\frac{dg}{g}$ is a formal integrating factor for $\tilde{\omega}$. As such, $\tilde{\eta}=f \tilde{\omega}$, for $f$ a formal meromorphic first integral of $\F_{\tilde{\omega}}$. This foliation is reduced and has only two formal separatrices, in particular if one takes  a reduced fraction $f=f_1/f_2$ with $f_i$ formal holomorphic, $f_1$ or $f_2$ is a unit and we have a holomorphic formal first integral, $f$ or $f^{-1}$. If this is not constant, by \cite[Th\'eor\`eme A]{MR608290}, we obtain a non constant converging holomorphic first integral for $\F_{\tilde{\omega}}$, contradicting infinite holonomy. Hence $\tilde{\eta}=\alpha \tilde{\omega}$ for some $\alpha\in \C$ and $\phi^*\eta=\alpha\tilde{\omega}-\frac{dg}{g}-\frac{du}{u}$. One concludes the residue and polar order for $\eta$ along $E_i$ are those of $\alpha\tilde{\omega}-\frac{dg}{g}$ along $\phi^*E_i$.

The argument in the next paragraphs relies on the formal classification of germs of foliations
with reduced singularities as recalled in the first chapter of \cite{MR3328860}.

If $\lambda\in \R\setminus \Q$ then the foliation is formally linearizable and we are
in the above situation with $g=xy$, $\tilde{\omega}=\lambda\frac{dx}{x}-\frac{dy}{y}$. The holonomy $y\mapsto e^{2i\pi\lambda}y$ has infinite order. We conclude $\eta$ has simple poles.

If $\lambda=-p/q$ with $p,q \in \mathbb N$, $\gcd(p, q)=1$,  then  either the foliation  is linearizable or, for some $\tau\in \C, k\in \mathbb N^*$, we have a formal normal form given by
$\tilde{\omega}=d(\frac{1}{(x^py^q)^k})+p(\tau-1)\frac{dx}{x}+q\tau\frac{dy}{y}$. In the latter case,  we are in the situation of the first paragraph with $g=(x^py^q)^kxy$. The holonomy $h:y\mapsto e^{2i\pi\lambda}(y+y^{k+1}+...)$ has infinite order because $h^{nq}$ is conjugate to some map $y\mapsto y+y^{k+1}+...$, for every $n>0$. If $\alpha\neq 0$ both poles have order $\geq2$ and $m_2-1=-\lambda (m_1-1)$. If $\alpha=0$ both poles are logarithmic, hence $m_2-1=-\lambda (m_1-1)$.

If $\lambda=0$ then we have a formal normal form $\tilde{\omega}=\frac{dx}{x}+\tau\frac{dy}{y}+d(\frac{1}{y^k})$. This case is also handled thanks to the first paragraph with $g=y^kxy$, as above, the holonomy $y\mapsto(y+y^{k+1}+...)$ has infinite order. One finds $m_1=k+1$ unless $\alpha=0$, in which case $m_1= 1$ and the residues are $(-k-1,-1)$. In any case $m_2\leq1$.

If $\lambda=-p/q$ with $p,q \in \mathbb N$, $\gcd(p,q)=1$ and,  in some formal coordinates, the foliation is given by $\tilde{\omega}=p \frac{dx}{x}+q\frac{dy}{y}$, the local pair $(xy\tilde\omega,\eta)$ must satisfy $\eta=g(p \frac{dx}{x}+q\frac{dy}{y})-\frac{dx}{x}-\frac{dy}{y}$, with $g=F^{\pm1}$ for a formal holomorphic first integral $F$. After \cite[Th\'eor\`eme A]{MR608290}, $F=f(x^py^q)$ for some $f\in \C[[t]]$. If $g=F$ both poles are logarithmic. If $g=F^{-1}$, $F=v(x^py^q)^k$, for a unit $v$ and $k\geq 0$. If $k=0$ we are in the previous case, otherwise $m_1=pk+1$ and $m_2=qk+1$, whence the conclusion.
\end{proof}

\begin{lemma}\label{disjoint}
Let $\G$ be a reduced foliation on a smooth complex surface which is transversely affine with \emph{logarithmic} connection $\nabla$. Assume $\nabla$ has  real residues and
decompose the divisor $\Res(\nabla)$ as $\Res(\nabla)_{>-1}+\Res(\nabla)_{\leq -1}$ where the second term is the part with coefficients $\leq -1$. Then
the supports of these two divisors are disjoint.
\end{lemma}
\begin{proof}
Let $s$ be an intersection of two polar components of $\nabla$. The foliation is given by a local one form $\lambda y dx-xdy+\ldots$ around $s$.
The logarithmic connection has a connection matrix $\eta=r_1f(x,y)\frac{dx}{x}+r_2g(x,y)\frac{dy}{y}$, where $f,g$ are holomorphic with $f(0,y)\equiv1$ and  $g(x,0)\equiv1$. Comparison of
$d\omega=(\lambda+1)dy\wedge dx+\ldots$ with  $\omega\wedge\eta=(r_1+\lambda r_2)dx\wedge dy+\ldots$  yields the identity  $(r_1+1)=-\lambda(r_2+1)$. If $\lambda \neq 0$, then
$r_1=-1$ is tantamount to  $r_2=-1$. If $r_2+1\neq0$ then one must have $\lambda \in \R$ and hence  $\lambda<0$, in particular $(r_1+1)(r_2+1)>0$.  If $\lambda=0$, by Lemma~\ref{chains}, $r_1,r_2\leq-1$.
In any case ($r_1\leq-1$ and $r_2\leq -1$) or ($r_1>-1$ and $r_2> -1$).
\end{proof}

\begin{lemma}\label{negdef}
Assumptions as in Lemma \ref{disjoint}.
The intersection form is negative definite on the support of $\Res(\nabla)_{>-1}$.
\end{lemma}
\begin{proof}
By Proposition~\ref{P:chave}, there exists a generically finite morphism $\pi :Y\rightarrow X\setminus \Res(\nabla)_{\leq-1}$ and a morphism $Y\to A$ that contracts $\pi^{-1}(\Res(\nabla)_{>-1})$.
Hence, by Grauert's criterion, the intersection form is negative definite on  $\pi^{-1}(\Res(\nabla)_{>-1})$. For a compact divisor $C$ in $A$, we have $\pi^*(C)\cdot\pi^*(C)=\ell\cdot C\cdot C$, where $\ell$ is the cardinality of the general fiber of  $\pi$. Hence, if $C$ is non trivial with $\supp(C)\subset \supp(\Res(\nabla)_{>-1})$,  as $\pi^*(C)\cdot\pi^*(C)<0$, we have $C\cdot C<0$. \end{proof}

\subsection{Proof of Theorem \ref{THM:C} for virtually transversely Euclidean foliations}\label{algo closed}
Assume  $\F$ is a degree  $d\geq 2$ virtually transversely Euclidean foliation on $\P^2$ without rational first integral.  Thus there exists a flat logarithmic connection with rational residues $\nabla$ which defines a transverse structure for $\mathcal F$. As in the proof of Theorem \ref{THM:A} we will decompose $\Res(\nabla)$ as $\Res(\nabla)_{>-1}+\Res_{\leq-1}(\nabla)$, with $\Res(\nabla)_{>-1}$ supported on the union of the components which have residues greater than $-1$.

We will show how to algorithmically  find a bound $B$ for the degree of $\supp(\Res(\nabla))=(\nabla)_{\infty}$. Once this is done we can compute all the $\F$-invariant curves of  degree no larger than $B$. If there are $m$ such algebraic $\F$-invariant curves and they are defined by $m$ irreducible homogeneous polynomials $f_1, \ldots, f_m$ then the existence of a transversely affine structure with logarithmic poles and rational residues is equivalent to  the existence  of a solution $(\lambda_1, \ldots, \lambda_m) \in \mathbb Q^m$  for  the  system of linear equations
\begin{align*}
 d\omega &= \omega \wedge \left( \sum_{i=1}^m  \lambda_i \frac{df_i}{f_i} \right) \, ,
\end{align*}
where $\omega$ is a homogeneous $1$-form on $\mathbb C^3$ defining $\mathcal F$.
To conclude the proof of Theorem~\ref{THM:C} for this case, we will explain in Proposition~\ref{prop:bound euclidean} that the bound $B$ obtained from our algorithm is smaller than a quantity that depends only on the three following data: the degree of $\F$, the number of singularities of a desingularisation of $\F$ and the corresponding Camacho-Sad invariants.

We start explaining our algorithm. On the one hand, a bound for the degree of $\supp(\Res(\nabla)_{>-1})$ is given
by Theorem \ref{T:transEuclidean}. On the other hand, using  Proposition~\ref{P:residues}, we can
write
\begin{align}\label{eq:res}
    \deg(\supp(\Res_{\leq-1}(\nabla))) \le - \deg \Res_{\leq-1}(\nabla) = \deg( N\F) + \deg \Res(\nabla)_{>-1}.
\end{align}
 Therefore in order to bound the degree
 of $\supp(\Res_{\leq-1}(\nabla))$ it suffices to bound the degree of the divisor $\Res(\nabla)_{>-1}$.

To this end, compute $\pi : (X,\mathcal G)\rightarrow (\P^2,\F)$  a reduction of singularities of $\mathcal F$ and  let $\mathcal C$ be the set of irreducible $\mathcal G$-invariant algebraic curves $C$ on $X$ satisfying $C \cdot \pi^* \mathcal O_{\mathbb P^2}(1) \le 12 (d-1)$. Notice that $\mathcal C$ consists of the strict transforms of irreducible $\mathcal F$-invariant algebraic curves of degree at most $12(d-1)$ and the $\mathcal G$-invariant irreducible components of the exceptional divisor of $\pi$.

Let $\nabla_{\G}$ denote the strict transform of $\nabla$. In  particular  $\nabla_{\G}$  is a logarithmic connection
on $N\G$ with rational residues. As we did for $\Res(\nabla)$, we decompose $\Res(\nabla_{\G})$ as
$\Res(\nabla_{\G})_{>-1} + \Res(\nabla_{\G})_{\le -1}$. The coefficients of $\Res(\nabla_{\G})_{>-1}$ along
the strict transform of $\F$-invariant curves coincide with those of $\Res(\nabla)_{>-1}$ on the corresponding curves. Therefore it suffices to bound the coefficients of $\Res(\nabla_{\G})_{>-1}$  to conclude.

For that sake we are going to consider all  subsets   of curves in $\mathcal C$, including the empty set. According to Theorem \ref{T:transEuclidean},  one of these subsets coincides with the support of $\Res(\nabla_{\G})_{>-1}$. For each choice of subset, say $C_1, \ldots, C_k$, we first check if the intersection matrix $(C_i \cdot C_j)$ is negative definite in order to be in accordance with Lemma \ref{negdef}. If it is not the case we discard this subset. Otherwise, we pretend that $\Res(\nabla_{\G})_{>-1}$ is supported on this subset and  use the negative definiteness
to determine the coefficients $(r_1, \ldots, r_k)$  of $\Res(\nabla_{\G})_{>-1}= \sum r_i C_i$. Indeed, it suffices to compute the integers $\alpha_j= N\G\cdot C_{j}$ and plug them into the system of linear equations
\begin{align}\label{system:euclidean}
    \sum_{i=1}^k r_i C_i \cdot C_j = \alpha_j, \, \quad   j \in \{ 1, \ldots, k \}\, .
\end{align}
The left hand side coincides with $N\G \cdot C_j$ thanks to Lemma \ref{disjoint}. The negative definiteness
of the matrix $(C_i \cdot C_j)$ guarantees the existence of a unique solution. Varying through all subsets of $\mathcal C$ we obtain a finite number of possibilities for $\Res(\nabla_{\G})_{>-1}$ and, consequently, a finite number of possibilities for $\deg \Res(\nabla)_{>-1}$.

\begin{prop}\label{prop:bound euclidean}
With the hypotheses and notation of Lemma~\ref{lem:bound}. Let $S$ be the finite set of solutions of linear systems $Ax=b$ such that
\begin{itemize}
\item $A\in M_k(\Z)$ is symmetric negative definite, $k \leq 2n$,
\item  The coefficients of $A$ satisfy $\vert A_{ij} \vert \leq \Gamma_4$, $1\leq i, j \leq k$,
\item the coefficients of $b$ satisfy $b_i\in \Z$,  $\vert  b_i \vert \leq n\Gamma_3+\Gamma_4$, $1\leq i \leq k$.
\end{itemize}
If $\F$ is virtually transversely Euclidean and $M$ is the maximal modulus of entries of elements of $S$,
then
\begin{bound}\label{bound:I}
\deg (\nabla)_{\infty}\leq 12(d-1)(M+1) +d+2.
\end{bound}
\end{prop}

\begin{proof}
Let $C_1,\ldots,C_k$ be the irreducible components of $\Res(\nabla_{\G})_{>-1}$ and $\Res(\nabla_{\G})_{>-1}=\sum r_iC_i$. From Lemma~\ref{negdef}, one has $C_i^2<0$, $i=1,\ldots,k$. In particular any $C_i$ contains a singularity of $\G$. As $\G$ is reduced, any singularity is contained in at most two invariant curves, so that $k\leq 2n$.
In addition to Lemma~\ref{lem:bound}, this shows the associated system~(\ref{system:euclidean}) satisfies the above conditions and $(r_i)$ belongs to $S$.
One has $\Res(\nabla)_{>-1}=\sum r_i \hat{C}_i$, where summation is restricted to indices for which the projection $\hat{C}_i\subset \P^2$ is a curve. By Theorem~\ref{T:transEuclidean}, the sum of the degrees of the curves $\hat{C}_i$ is smaller or equal to $12(d-1)$. From $\vert r_i\vert \leq M$, one concludes $\vert\deg \Res(\nabla)_{>-1}\vert\leq 12(d-1)M$.
The inequality (\ref{eq:res}) then shows
\[ \deg (\nabla)_\infty=\deg \supp(\Res(\nabla)_{>-1})+\deg \supp(\Res(\nabla)_{\leq-1})\leq 12(d-1)(M+1) +d+2. \]\end{proof}

Moreover, using Cramer's formulas for the inverse, one easily estimates  $M$:
\begin{bound}
    M \leq 2n(n\Gamma_3+\Gamma_4)\Gamma_4^{2n-1}(2n-1)! \, .
\end{bound}
This concludes the proof of Theorem \ref{THM:C} in the particular case of transversely Euclidean foliations.

\subsection{Conclusion of the proof of Theorem \ref{THM:C}} \label{algo Riccati}
Let $\F$ be a transversely affine foliation on $\P^2$ with  $\deg(\F)\geq 2$. Assume that $\F$ is not virtually transversely Euclidean. In particular, $\F$ is not algebraically integrable.

Our task now does not reduce to bounding the degree of the support of $(\nabla)_{\infty}$. We also have to bound the coefficients of the irreducible components in it.  Let $(D_i)_{1\leq i\leq m}$ be  algebraic curves defined by irreducible homogeneous polynomials $(f_i)_{1\leq i\leq m}$ Let $(n_i+1)_{1\leq i\leq m}$ be positive integers. The existence of a transversely affine structure for $\F$ with polar divisor
$\sum(n_i+1)D_i$  is equivalent to the existence   of a solution $(\lambda_1, \ldots, \lambda_m,g) \in \mathbb C^m \times \mathbb C[x,y,z]$  for  the  system of linear equations
\begin{align}\label{system:structure}\left \lbrace
\begin{array}{rl}
 d\omega &= \omega \wedge \left( \sum_{i=1}^m  \lambda_i \frac{df_i}{f_i}  + d \left( \frac{g}{\prod_{i=1}^m {f_i}^{n_i}} \right)\right)  \\
 \deg(g) &= \sum_{i=1}^m n_i \deg(f_i) \, ,
 \end{array}
 \right .
\end{align}
where $\omega$ is a homogeneous polynomial $1$-form on $\mathbb C^3$ defining $\mathcal F$.

We know from Theorem \ref{T:structure} that
$\F$  is the pull-back of a transversely affine Riccati foliation $\mathcal R$ on a ruled surface $S$ under a rational map $p: \P^2 \dasharrow S$. As in Section \ref{algo closed},  compute $\pi: (X,\mathcal G)\rightarrow (\P^2,\F)$  a reduction of singularities of $\F$. Denote $f:\P^2\dasharrow \P^1$ the pull-back by $p$ of the adapted fibration of $\calR$, which we may and do suppose to have irreducible general fiber. Denote also $\tilde{f}=f\circ \pi$, $\tilde{p}=p\circ \pi$. By Lemma \ref{lemme fibration}, the rational map $\tilde{f}$ is a genuine fibration.

Decompose the polar divisor of $\nabla_{\mathcal G}$ as $H+V$, where the irreducible components of $H$ dominate the basis of $\tilde{f}$ and no component of $V$ has this property. Let $\sigma$ be the unique  $\calR$-invariant algebraic curve in $S$  that dominates the basis of the reference fibration, cf. Lemma \ref{L:propR}. Notice $\tilde{p}(H)=\sigma$ so that $H$ is non empty and $\nabla_{\mathcal G}$ has only simple poles in the components of $H$. A simple computation shows that the residues along the irreducible components of $H$ have the form $-k-1$, $k\in \mathbb N^*$.

Our goal is to bound the intersection number $(\nabla)_{\infty} \cdot \mathcal O_{\mathbb P^2}(1)=(H+V) \cdot \pi^* \mathcal O_{\mathbb P^2}(1)$. We will make use of the following lemma.

\begin{lemma}\label{L:signatureW}
Notation as above. The Neron Severi group of $X$ is denoted $NS(X)$. The following assertions hold true.
\begin{enumerate}
\item If  $W' \subset NS(X)\otimes \R$ is the vector subspace spanned by the classes of the irreducible components of $\supp(V)$ then the intersection form on $W'$  is negative semi-definite with kernel of dimension $1$, spanned by a fiber of $\tilde{f}$.
\item If $W\subset NS(X)\otimes \R$ is the vector subspace spanned by  $W'$  and by the Chern class of $N\G$ then  the  intersection form on $W$ is non-degenerate.
\end{enumerate}
\end{lemma}
\begin{proof}
(1) Since we are assuming $\F$ is not virtually transversely Euclidean, Lemma~\ref{L:only dicritic} guarantees that the support of $V$ contains the support of a fiber of $\tilde{f}$. Once one observes that the classes of any two fibers are the same in $NS(X)$, the result follows from Zariski's lemma about the intersection of fiber components of a fibration, cf. \cite[Chapter III, Lemma 8.2]{MR2030225}.

(2) We can choose a basis $(F,e_1,\ldots, e_r)$ for $W'$ with $F$ equal to (the class of) a fiber, and $e_1, \ldots, e_r$ elements satisfying $e_i \cdot e_j = - \delta_{ij}$ (Kronecker delta)  and $F\cdot e_i = 0$. Let $D=\sum_i m_iH_i$ be the part of $-\Res(\nabla_{\G})$ supported on $H$, with $H_i$ irreducible, $m_i\geq2$.
By Proposition~\ref{P:residues}, we have  $D=N\G +w'$, where $w'\in W'$. In particular $D\in W$ and
we have $N\G \cdot F=D \cdot F$, because $F$ intersects trivially any element of $W'$. As a consequence, representing $F$ by a  smooth non-invariant fiber of $\tilde{f}$, one checks $N\G \cdot F\geq \sum m_i>0$. If $q$ is the matrix of the intersection form restricted to $W$  in the basis $(N\G,F,e_1,\ldots, e_{r})$, expansion  with respect to the first line shows $\det(q)=\pm (N\G \cdot F)^2\neq0$.
\end{proof}

Let $\mathcal C$ be the set of $\mathcal G$-invariant algebraic curves $C$ on $X$ satisfying $C \cdot \pi^* \mathcal O_{\mathbb P^2}(1) \le 6 (d-1)$. Theorem  \ref{T:Kod1} guarantees that the support of $V$ is formed by members of $\mathcal C$.
As in the proof of Lemma \ref{L:signatureW}, write $\sum_i m_iH_i = D = N\G +w'$, where $w'\in W'$. Notice that the divisor $H$  appearing in the decomposition $(\nabla)_{\infty} = H+V$  is a  reduced divisor since the {\it horizontal} poles of $\nabla$ are logarithmic. As $m_i\geq 2$, it thus suffices to bound $D\cdot \pi^* \mathcal O_{\mathbb P^2}(1)$ to bound  $H\cdot \pi^* \mathcal O_{\mathbb P^2}(1)$.  To achieve this, consider all the subsets of $\mathcal C$  which have semi-definite negative intersection matrix with one dimensional kernel; and  the kernel is not orthogonal to $N\G$. Restrict furthermore to subsets for which the kernel is generated by an element with non-negative integral coefficients.

Let $C_1, \ldots, C_k$ be one such subset and let $F_0$ be a generator of the kernel, with relatively prime non-negative integral coefficients. We explain below how the algorithm succeeds assuming this set yields a basis of $W'$. The reader will easily find out how the algorithm could stop in case this set is not a basis of $W'$.

First observe that  $\supp(V)\cdot \pi^* \mathcal O_{\mathbb P^2}(1)$ is smaller than the degree of the divisor $\tang$ of tangencies between the foliation $\F$ and  a foliation $\mathcal H$ defined by a pencil of degree $6(d-1)$ curves. One can see $\tang$ has degree not bigger than $13(d-1)$, since $\deg(\mathcal H)\leq 12(d-1)-2$ and $\deg(\tang)=\deg(\mathcal H)+\deg(\F)+1$. In particular, we get the estimate $\sum C_i \cdot \pi^* \mathcal O_{\mathbb P^2}(1)\leq 13(d-1)$.

Remembering $N\G \cdot \pi^* \mathcal O_{\mathbb P^2}(1)=d+2$, bounding $D\cdot \pi^* \mathcal O_{\mathbb P^2}(1)$ is then reduced to bounding the coefficients $\beta_i$ in the decomposition $D=N\mathcal G+\sum \beta_i C_i$. To this aim, first observe  that $D \cdot C_i$ is non-negative  and bounded by $\max_{j} m_j \cdot Z(\mathcal G, C_i)$.
To bound $\max_j m_j$, choose $F_1 = \lambda \cdot F_0$ in such a way that
$F_1 \cdot \pi^* \mathcal O_{\mathbb P^2}(1) = 6(d-1)$.
The class of $F_1$ will be a rational multiple of the class of  a fiber of the fibration $\tilde{f}$,
with corresponding factor bigger or equal to $1$. Also observe $\lambda\leq 6(d-1)$, for $F_0 \cdot \pi^* \mathcal O_{\mathbb P^2}(1)\geq 1$.
Using that any element in $W'$ intersects trivially the fibers of $\tilde{f}$, we deduce
\begin{equation}\label{bound m_i}
    2\leq\sum_i m_i \le  \sum_i m_iH_i \cdot F_1 = N\mathcal G \cdot F_1
    \, .\end{equation}
This allows to  bound the integers $M_0,M_j$ in the following square system of  linear equations for $(1,\beta_1,\ldots,\beta_k)$.
\begin{align}\label{sys:D}
 \left \lbrace  \begin{array}{rl}
    N\G \cdot N\G+\sum_i\beta_i (C_i\cdot N\G)  & = M_0 , \\
    N\G \cdot C_j+\sum_i\beta_i (C_i\cdot C_j)   & = M_j,  \quad \text{ for }  j \in \{1,\ldots, k\} \, .
    \end{array} \right.
 \end{align}
Indeed, with the notation of Lemma~\ref{lem:bound}, one has
\begin{align}\label{unpromoted}
\left \lbrace \begin{array}{rclc}\vert D \cdot C_i  \vert &\leq& Z(\F,C_i) N\G \cdot F_1&\\
\Gamma_2N\G \cdot F_1 &\leq& D\cdot N\G=\sum m_iH_i\cdot N\G &\leq \left(n(\Gamma_3+2)+\Gamma_1\right)N\G \cdot F_1.
\end{array} \right.
\end{align}
The second line is obtained combining  Lemma~\ref{lem:bound}-(\ref{lem:bound:iii}) and $(\ref{bound m_i})$.
Using Lemma~\ref{lem:bound F_0} below, $N\G\cdot F_1\leq 6(d-1)N\G\cdot F_0$ and $Z(\F,C_i)\leq n\Gamma_3$, one may promote these bounds to  coarser ones depending only on $d$, the number $n$ of singularities of $\G$ and the corresponding Camacho-Sad invariants.

 \begin{lemma}\label{lem:bound F_0} With the notation of Lemmas \ref{lem:bound} and \ref{L:signatureW}.
 If $C_1,\ldots, C_k$ is a basis of $W'$, then $k\leq 2n+1$ and
\begin{align} \label{bound:F_0} \vert N\G\cdot F_0\vert \leq (4n+2)(n\Gamma_3+\Gamma_4)\Gamma_4^{2n}(2n)!~~.
\end{align}
\begin{proof}
Let $F_0=\sum \alpha_i C_i$.
As the intersection form on $W'$ has a kernel of dimension $1$, none but one of of the curves $C_i$ may have zero self-intersection, hence $k\leq 2n+1$. Moreover, thanks to Lemma~\ref{lem:bound},  \[\vert C_i \cdot C_j\vert \leq \Gamma_4, ~~ 1\leq i \leq j \leq k.\]
Combined with \cite[Theorem 2]{MR0396605} applied to the system $\left(\sum_i \alpha_i C_i\cdot C_j=0\right)_j$, this yields \[\vert \alpha_i\vert \leq 2\Gamma_4^{k-1}(k-1)!\leq 2\Gamma_4^{2n}(2n)! ~.\]
From Lemma~\ref{lem:bound}-(\ref{lem:bound:iii}), we also have $\vert N\G\cdot C_i\vert \leq n\Gamma_3+\Gamma_4$.
The triangle inequality then yields the announced result.
\end{proof}
\end{lemma}
 Expressing the solution $(1,\beta_1,\ldots,\beta_k)$ of the linear system (\ref{sys:D}) by means of  Cramer's rule, and taking into
 account the integrality of the coefficients, one obtains an estimate for the coefficients $\beta_i$.  This provides a bound for $H \cdot \pi^* \mathcal O_{\mathbb P^2}(1)$.

  In view of proving Theorem~\ref{THM:C},  observe that if $\frac{\Gamma_5}{6(d-1)}$ is the upper bound in (\ref{bound:F_0}), a promoted version of (\ref{unpromoted}) is
 \[\vert D\cdot N\G \vert \leq (n\Gamma_3+\Gamma_4)\Gamma_5,~~ \vert D\cdot C_i\vert \leq (n\Gamma_3+\Gamma_4)\Gamma_5,~~i=1,\ldots, k\, , \]
 and one easily deduces, setting  $\tilde{\Gamma}_4=\max(N\G^2,n\Gamma_3+\Gamma_4)$,
\[ \vert \beta_i\vert \leq (2n+2)(n\Gamma_3+\Gamma_4)\Gamma_5\tilde{\Gamma}_4^{2n+1}(2n+1)!~~, \]
so that
 \begin{bound} \label{bound:H}H \cdot \pi^* \mathcal O_{\mathbb P^2}(1)\leq d+2+13(d-1)(2n+2)(n\Gamma_3+\Gamma_4)\Gamma_5\tilde{\Gamma}_4^{2n+1}(2n+1)!~.
\end{bound}
In practice, a better estimate is obtained from (\ref{unpromoted}) and computation of $N\G\cdot F_1$.

 At this point, we have a bound for the degree of all irreducible components in the support of $(\nabla)_{\infty}$, \textit{e.g.}  the maximum of $6(d-1)$ and the right hand side of (\ref{bound:H}).
  Identify all $\G$-invariant algebraic curves $C$ with $C\cdot \pi^*\mathcal{O}_{\P^2}(1)$ not bigger than this bound, denote $\mathcal P$ the set of such curves. We now proceed to bound the coefficients of the effective divisor $(\nabla)_{\infty}$.
 From Lemma~\ref{L:propR} (\ref{L:propR:6})-\ref{L:propR:6:iii},  any fiber of $\tilde{f}$ which contains multiple poles of $\nabla_{\G}$ is $\G$-invariant  and one of its components contains a singularity through which passes a non algebraic (perhaps formal) separatrix.
The polar order of $\nabla_{\G}$ at this non algebraic separatrix  is zero. Hence in any fiber of $\tilde{f}$ containing a multiple pole of $\nabla_{\G}$, we have a singularity which is contained in only one curve of $\mathcal{C}$, say $\tilde{C}$. The other local (possibly formal) separatrix is not a pole of $\nabla_{\G}$ (non polar).
From Lemma~\ref{chains} and Milnor number computations, this allows to bound inductively the order of poles of $\nabla_{\G}$ in any $\G$-invariant component which is connected to $\tilde{C}$ by a sequence of elements of $\mathcal{C}$. The components of the fiber are among such curves.

Assign in this way a bound $b_C$ for any irreducible curve $C\in \mathcal C$ which is connected to a non polar separatrix by a sequence of  curves belonging to $\mathcal C$. If there are various
such sequences, take the smallest bound.
In this manner, all the curves $C$ that are contained in invariant fibers of $\tilde{f}$ and are multiple poles of $\nabla_{\G}$ are endowed with a bound $b_C$.  For the remaining elements of $\mathcal{P}$, assign the trivial  bound $b_C=1$.  This bounds the coefficients of $(\nabla_{\G})_{\infty}=\sum_{C\in \mathcal{P}} a_C C$, namely $a_C\leq b_C$. This allows to find the transverse structure of $\F$ by solving finitely many systems (\ref{system:structure}). This finishes the description of our algorithm, we now conclude the proof of Theorem~\ref{THM:C}.

The coefficients $a_C$ can be bounded uniformly in terms of the Camacho-Sad indices and the degree of the foliation as follows.
\begin{lemma}\label{chains:2} With the assumptions and notation of Lemma~\ref{lem:bound}.
Let $C_1,\ldots, C_k$ be a chain of $\G$-invariant algebraic curves, such that $C_1$ bears a singularity $q$ free of other algebraic separatrix. Let $\Lambda=\max( \{-i_{\ell}\vert i_{\ell}\in \Q \}\cup\{0\})$. For a given transverse affine structure for $\G$, denote $a_i$ the polar order along $C_i$.
Then, for all $i$,
\[\left \lbrace
\begin{array}{l}a_i\leq \Gamma_3\frac{\Lambda^i-1}{\Lambda-1}\mbox{ if } \Lambda >1,\\
a_i\leq i\Gamma_3\mbox{ if } \Lambda=1,\\
a_i\leq \Gamma_3\mbox{ if } \Lambda<1.\\
\end{array}
\right.
\]
Observe that, in any case, the given bound for $a_k$ is also valid for $a_i$.
\end{lemma}
\begin{proof}
First, if $\Lambda<1$, then every singularity of $\G$ is a saddle-node or has irrational quotient of eigenvalues, in particular, from  Lemma~\ref{chains}, $a_i\leq\Gamma_3$.

Otherwise $\Lambda\geq 1$, we introduce the second formal separatrix $C_0$ through $q$, which has polar order $a_0=0$.
It follows from Lemma~\ref{chains}, that $a_{i}\leq g(a_{i-1}),~i=1,\ldots,k$, where $g(x)=\Lambda x+\Gamma_3$.
The function $g$ is increasing on $\R$ and a  bound $b_{i-1}$ for $a_{i-1}$ induces the bound $b_i=g(b_{i-1})$ for $a_i$. The sequence of bounds $b_0=a_0, b_i=g(b_{i-1})$ is the announced one.
\end{proof}
In our context, we apply   Lemma~\ref{chains:2} to chains of curves contained in fibers of $\tilde{f}$, hence of length $k\leq 2n+1$.
 With the notation of Lemma~\ref{chains:2}, set $\Gamma_6=\Gamma_3\frac{\Lambda^{2n+1}-1}{\Lambda-1}$ if $\Lambda>1$, $\Gamma_6=(2n+1)\Gamma_3$ if $\Lambda=1$, and $\Gamma_6=\Gamma_3$ otherwise. We infer that the polar multiplicities $a_C$ in $V$ are not bigger than $\Gamma_6$. As a consequence,
\begin{bound}\label{bound:IV}
V\cdot\mathcal{O}_{\P^2}(1)\leq 13(d-1)\Gamma_6.
\end{bound}
Combined with (\ref{bound:H}), this concludes the proof of Theorem~\ref{THM:C}.\flushright
\bibliographystyle{amsalpha}
\bibliography{biblio}

\end{document}